%% file: main.tex
\documentclass{article}
\usepackage{PRIMEarxiv}

\usepackage[utf8]{inputenc} 
\usepackage[T1]{fontenc}    
\usepackage{hyperref}       
\usepackage{url}            
\usepackage{booktabs}       
\usepackage{amsfonts}       
\usepackage{nicefrac}       
\usepackage{microtype}      
\usepackage{lipsum}
\usepackage{fancyhdr}       
\usepackage{graphicx}       
\graphicspath{{media/}}     
\usepackage{amsmath,bm}
\usepackage{hyperref}       
\usepackage{cleveref}
\usepackage{amsfonts}       
\usepackage{graphicx}
\usepackage{multicol}
\usepackage{multirow}
\usepackage{float}
\usepackage{listings}
\usepackage[english]{babel}
\usepackage{amsthm}
\usepackage{caption}
\usepackage{subcaption}
\usepackage{xcolor}
\usepackage{mathtools}
\usepackage{authblk}

\input{content/symbols} 

\input{content/def_math} 
\title{DOSnet as a Non-Black-Box PDE Solver: When Deep Learning Meets Operator Splitting}

\author[1, 2]{\small Yuan Lan}
\author[1*]{\small Zhen Li}
\author[1*]{\small Jie Sun}
\author[2,3*]{\small Yang Xiang}

\affil[1]{\footnotesize Theory Lab, Huawei Technologies, Co. Ltd., Hong Kong SAR, China}
\affil[2]{\footnotesize Department of Mathematics, The Hong Kong University of Science and Technology, Clear Water Bay, Hong Kong SAR, China}
\affil[3]{\footnotesize Algorithms of Machine Learning and Autonomous Driving Research Lab, HKUST Shenzhen-Hong Kong Collaborative Innovation Research Institute, Futian, Shenzhen, China}
\affil[*]{Corresponding authors: lishen03@gmail.com, j.sun@huawei.com, maxiang@ust.hk}

\begin{document}
\maketitle
\begin{abstract}
Deep neural networks (DNNs) recently emerged as a promising tool for analyzing and solving complex differential equations arising in science and engineering applications. Alternative to traditional numerical schemes, learning-based solvers utilize the representation power of DNNs to approximate the input-output relations in an automated manner. However, the lack of physics-in-the-loop often makes it difficult to construct a neural network solver that simultaneously achieves high accuracy, low computational burden, and interpretability. In this work, focusing on a class of evolutionary PDEs characterized by having decomposable operators, we show that the classical ``operator splitting'' numerical scheme of solving these equations can be exploited to design neural network architectures. This gives rise to a learning-based PDE solver, which we name Deep Operator-Splitting Network (DOSnet). Such non-black-box network design is constructed from the physical rules and operators governing the underlying dynamics contains learnable parameters, and is thus more flexible than the standard operator splitting scheme. Once trained, it enables the fast solution of the same type of PDEs. To validate the special structure inside DOSnet, we take the linear PDEs as the benchmark and give the mathematical explanation for the weight behavior. Furthermore, to demonstrate the advantages of our new AI-enhanced PDE solver, we train and validate it on several types of operator-decomposable differential equations. We also apply DOSnet to nonlinear Schr\"odinger equations (NLSE) which have important applications in the signal processing for modern optical fiber transmission systems, and experimental results show that our model has better accuracy and lower computational complexity than numerical schemes and the baseline DNNs.
\end{abstract}


\section{Introduction}\label{sec:introduction}
\input{content/introduction}

\section{Results}\label{sec:result}
\input{content/method}

\input{content/examples}
\input{content/case_study}

\section{Methods}\label{sec:method}
\input{content/method_NCS}

\section{Conclusions}\label{sec:disscussion}
\input{content/discussions}

\section*{Acknowledgments}
The work of Yang Xiang was supported in part by HKUST IEG19SC04 and the Project of Hetao Shenzhen-HKUST Innovation Cooperation Zone HZQB-KCZYB-2020083.

\appendix
\input{content/appendix}

\bibliographystyle{unsrt}  
\bibliography{main}

\end{document}

%% file: content/symbols.tex
\newcommand{\bmx}{\bm{x}}

\newcommand{\cC}{\mathcal{C}}

\newcommand{\cL}{\mathcal{L}}

\newcommand{\cN}{\mathcal{N}}
\newcommand{\cO}{\mathcal{O}}

\newcommand{\cS}{\mathcal{S}}

\newcommand{\cW}{\mathcal{W}}
\newcommand{\cX}{\mathcal{X}}

\newcommand{\bbC}{\mathbb{C}}

\newcommand{\bbR}{\mathbb{R}}

\newcommand{\btheta}{\bm{\theta}}



%% file: content/def_math.tex
\newtheorem{proposition}{Proposition}

\DeclareMathOperator*{\argmin}{arg\,min\;}

%% file: content/introduction.tex
Evolutionary partial differential equations arise as mathematical models for evolving in a wide range of fields in science and engineering.
To simulate the evolution of such a real system by solving the associated evolutionary PDE is typically computationally expansive.
One engineering application is nonlinear compensation in optical fiber \cite{Li2008ElectronicPO, Ip2008CompensationOD, Asif2011DigitalBP},
which aims to recover the clean signal from the distorted signal induced by long-distance propagation in fiber. 
Mathematically, the compensation can be obtained by solving an inversed nonlinear Schr\"{o}dinger equation (NLSE). The challenge is how to obtain an accurate solution to the inversed NLSE with low computational complexity to meet the engineering needs.

One classical numerical approach to solving the inversed NLSE over a long distance is the split-step Fourier method (SSFM) \cite{Weideman1986SplitstepMF, Hardin1973ApplicationOT}, which is based on operator splitting \cite{strang1968construction}. 
In one calculation step, the operator splitting scheme splits the inversed NLSE into linear and nonlinear parts, then combines the analytical solution of each part to form the final solution. However, the idea of splitting brings the splitting error controlled by step size. 
 A smaller step size yields a smaller splitting error, nevertheless, results in more computation steps thus makes it impractical in real-time implementation \cite{Asif2011DigitalBP}. In the past two decades,  various works have been proposed to reduce the computational complexity in SSFM by balancing the step size and splitting errors \cite{Du2010ImprovedSC,Lin2010CompensationOT,Asif2011LogarithmicSB,Fougstedt2017TimedomainDB,Li2011ImplementationEN,Yan2011LowCD,Liang2015CorrelatedDB,Rafique2011ImpactOS}. 
Some works apply  well-designed combinations of the linear and nonlinear operators to pursue a larger step size meanwhile maintain an acceptable splitting error \cite{Asif2011LogarithmicSB, Du2010ImprovedSC, Lin2010CompensationOT}. In modified SSFM (M-SSFM) \cite{Lin2010CompensationOT}, they introduce a coefficient to shift the position of the nonlinear operator calculation point along with the optimization of linear operator. By balancing the combination between the linear and nonlinear operator, M-SSFM reduces the computational cost. Similarly, logarithmic SSFM (L-SSFM) \cite{Asif2011LogarithmicSB} achieves high computational efficiency by introducing a logarithmic step-size assignment for the linear and nonlinear operators according to the exponential decay of signal power. Another class of methods reduce the computational complexity by modifying the composition inside linear and nonlinear operators of SSFM  \cite{Fougstedt2017TimedomainDB,Li2011ImplementationEN,Rafique2011ImpactOS,Yan2011LowCD,Liang2015CorrelatedDB}. For example, Correlated-DBP (CBP) \cite{Li2011ImplementationEN, Rafique2011ImpactOS} modifies the original nonlinear operator by adding an additional correlation term with adjacent symbols. It considers the physical nature that the nonlinear distortion imprinted on one symbol is related to the pulse broadening of neighboring symbols, thus achieves better computation accuracy with a large step size. Besides that, Perturbation BP (PBP) \cite{Yan2011LowCD, Liang2015CorrelatedDB} applies the nonlinear perturbation analysis and adds a first-order perturbation term to nonlinear operators. Although those works reduce the computational complexity of SSFM, it remains challenging to fulfill the increasing demand for transmission rate and distance requirements in the engineering application. 

The last decade has seen a significant amount of research on solving partial differential equations (PDEs) using deep neural networks (DNNs) \cite{Beck2020AnOO, Raissi2019PhysicsinformedNN, Lu2019DeepONetLN, Han2018SolvingHP}. One of the typical DNN-based PDE solvers is differential operator approximation (DOA) \cite{Lu2019DeepONetLN, Li2021FourierNO}, in which the trained DNN serves as a surrogate operator for the PDE and infers the solution at the target time from arbitrary initial condition.
 The performance of DOA depends on the bias and variance of models rather than the step size in numerical schemes \cite{hastie2009elements}. This shows that DNNs can achieve large step-size evolution of PDEs.
 For nonlinear compensation tasks in optical fiber, some works have been proposed  using the idea of DOA \cite{fujisawa2021nonlinear,zhao2020low,kamalov2018evolution,hager_deep_2018, Hger2018NonlinearIM,fan_advancing_2020}. Among them, the ``black-box"-type models mitigate the nonlinear distortion directly from data using large amounts of parameters. To control the computational complexity of such models, weight pruning \cite{fujisawa2021nonlinear}, recurrent structure \cite{zhao2020low}, or shallow neural network \cite{kamalov2018evolution} are used. In contrast, the SSFM-based models design the neural network with prior knowledge from the inversed NLSE \cite{hager_deep_2018, Hger2018NonlinearIM,fan_advancing_2020} using much fewer parameters.  
They unroll the linear and nonlinear steps of SSFM as layers in a neural network. In the nonlinear step,  the nonlinear operator of SSFM is fixed as the activation function in the network. Meanwhile, convolution layers with small kernel sizes replace the dense linear operators in SSFM and the parameters of the convolutions are optimized to minimize the loss function. 
A limitation of the SSFM-based models is the predetermined initialization in convolutional layers.
Usually, those models use the truncation version of the original linear operators of SSFM as the initialization of convolution layers to accelerate the optimization. However, the predetermined initialization introduces extra truncation errors during training. Moreover, the predetermined initialization pre-specifies  the step size of the operator to be approximated by each layer of the network. Thus it loses the flexibility to learn the adaptive step size from data and limits the network’s performance.

In this paper, we develop an adaptive-step-size  SSFM-based neural network, \textit{Deep Operator Splitting Net (DOSnet)}, to solve the evolutionary PDEs.
We introduce the prior knowledge of the PDE in the design of the network by the dynamical structure, \textit{Autonomous Flow (Autoflow)}, which consists of several blocks with the same dimension for the input and output to mimic the evolution of the PDE. Compared to the standard neural networks, which transform the inputs into features in higher dimensions by large amounts of parameters, Autoflow has much fewer parameters by restricting the sizes of intermediate outputs.
Compared to the previous SSFM-based neural networks, we do not use the predetermined initialization that restricts the step size represented by each block. Instead, we apply random initialization for the parameters inside Autoflow. Thus, each block inside Autoflow learns an adaptive step-size operator. Each block of Autoflow is an operator splitting block (OSB) which contains a series of linear layers and the particular nonlinear activation functions from the original PDE. This special activation function from the PDE reflects the properties of the PDE and increases the expressive power of the network. See Fig.1 for the structure of AutoFlow and OSB of our DOSnet and comparison with the structure of a standard neural network.

Experimentally, we apply this method to several different classes of PDEs. We first validate the Autoflow structure on two types of linear PDEs as toy examples. We observe that the transition states inside Autoflow follow the true physical law in the linear models. For this observation, we give a mathematical explanation based on the behavior of weights of Autoflow during training. Then, we test our algorithm on the Allen-Cahn equation. Our result shows that DOSnet has high accuracy in predicting the solutions during the long-term evolution.
Moreover, the intermediate outputs of Autoflow represent states on the true trajectory of the evolution of the system, which cannot be observed in standard neural networks. We apply our DOSnet to the nonlinear Schr\"{o}dinger equation for the engineering application of nonlinear compensation in optical fiber. For a single-channel single-polarization 16 QAM system over optical fiber of $20 \times 80$ km under 100 GBaud/s transmission rate, experimental results show that our model has better accuracy and lower computational complexity than numerical schemes and the baseline DNNs.

Finally, we briefly review the classical operator splitting method \cite{strang1968construction, quispel2002splitting}. 
In this paper, we focus on the autonomous evolutionary PDEs with decomposable operators. It can be generally written as
\begin{equation}\label{eq:PDE}
    u_{t} = \cL u + \cN u,
\end{equation}
where $\mathcal{L}$ and  $\mathcal{N}$ are the linear and nonlinear part of the PDE, respectively. With some given initial condition $u\left(\bmx,0\right) = u_0$, the exact solution can be written as $u(\bmx,t) =  e^{t(\cL+\cN)}u_0$. One numerical scheme that utilizes the decomposable property to solve Eq.\eqref{eq:PDE} is operator splitting \cite{strang1968construction,quispel2002splitting}. The key idea of operator splitting is to split Eq.\eqref{eq:PDE} into two simpler sub-equations:
\begin{equation}\label{eq:splitting-PDE}
\left\{  
\begin{aligned}
    &u_{t} = \cL u \\
    &u_{t} = \cN u,
    \end{aligned}
\right.
\end{equation}
and the two sub-equations are solved in  closed forms on small time intervals. 
Specifically, suppose that we want to solve Eq.\eqref{eq:PDE} up to time $T$. 
The time interval $[0,T]$ is first divided into sub-intervals $0=t_{0}<t_{1}<\dots<t_{N}=T$ such that $t_{n}-t_{n-1}=\tau$ for all $n=1,2,\dots, N$. On each sub-interval, instead of solving the entire equation Eq.\eqref{eq:PDE}, one can solve the two equations  in Eq.\eqref{eq:splitting-PDE} successively. Hence the final solution $u(\bmx,T)$ is expressed as a stack of alternating compositions of linear and nonlinear evolution operators

    \begin{align} 
    u(\bmx,T)
    \approx  e^{\tau\cN} e^{\tau\cL} \dots e^{\tau\cN} e^{\tau\cL}   u_{0} \label{eq:gcOSN},
    \end{align}
The accuracy of the operator splitting depends on both the arrangement of the linear and nonlinear operators and the temporal step size $\tau$. For example, the plain splitting $e^{\tau\cN} e^{\tau\cL}$ 
has an error $O(\tau)$, while a more symmetric Strang splitting \cite{strang1968construction} $e^{\frac{\tau}{2}\cL} e^{\tau\cN}e^{\frac{\tau}{2}\cL}$ results in an error $O(\tau^{2})$. The challenge of  the classical operator splitting is how to balance the solution accuracy and computational cost.




\begin{figure}[htbp]
	\centering
	\includegraphics[width=\linewidth]{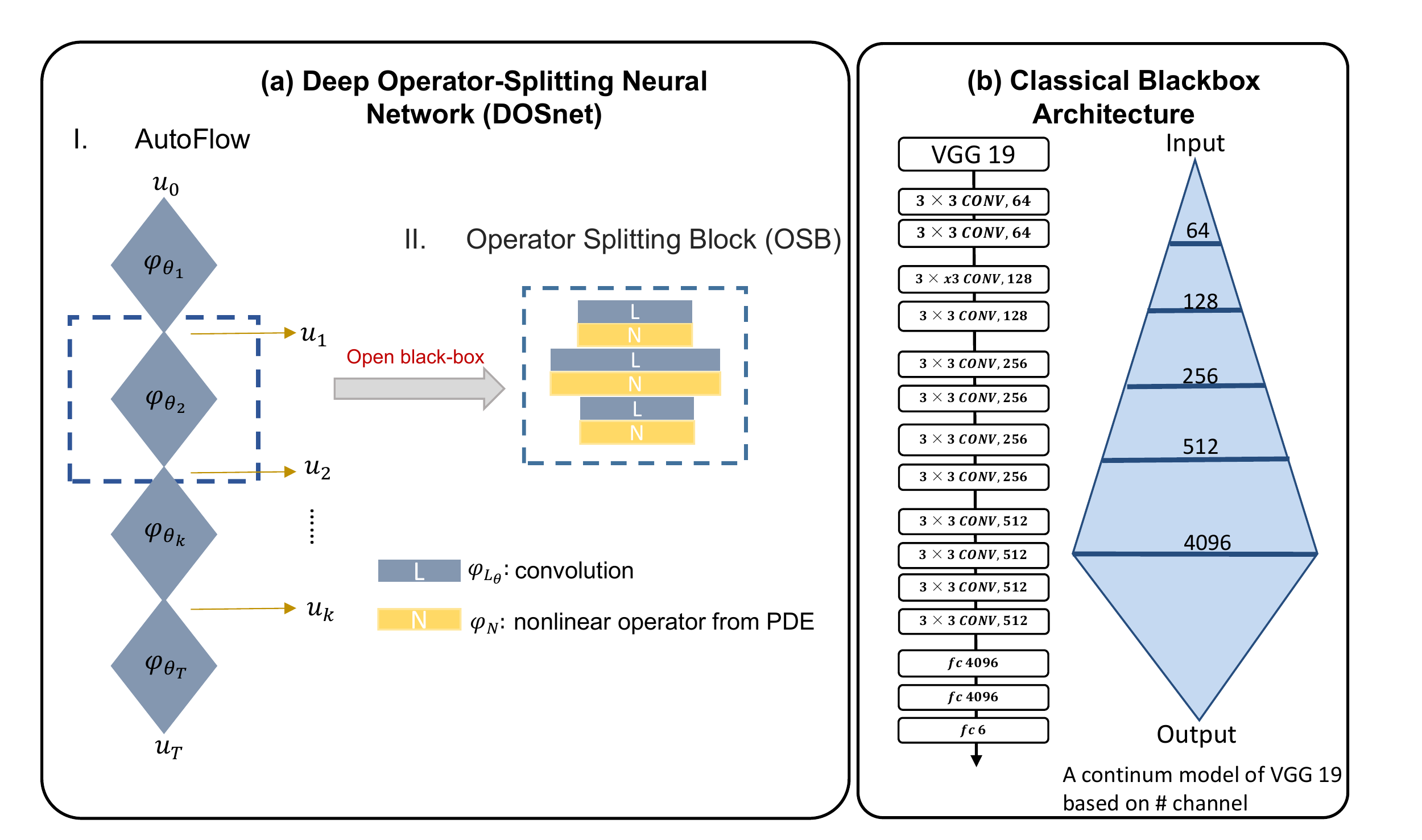}
	\caption{Architecture of DOSnet as depicted in (a), where network structure is designed by a hierarchy of two level: in the coarse-grained level, an Autoflow structure is used to ensure that each block functions as an autonomous flow with identical input and output dimensions; whereas in the finer level inside each block, a cascade of learnable linear layer with  nonlinear layers is imposed whose nonlinear functioning comes from the underlying equation. In contract, a standard DNN model such as the continum model of VGG19 \cite{simonyan2014very} as shown in (b) is characterized by completely flexible layers between the input and output: in this particular example, 
	layers maps the original input into the higher dimensional feature space (such as 64 dim, 128 dim or higher).}
	\label{fig:modelpipline}
\end{figure}

%% file: content/method.tex
\subsection{Deep Operator-Splitting Neural Network (DOSnet)}
We propose an adaptive-step-size, low computational burden \textit{Deep Operator Splitting Net (DOSnet)} to solve Eq.\eqref{eq:PDE} over a long time $T$. The key to our method is a general dynamical architecture we call Autonomous Flow (Autoflow), describing the action of
the evolution operators. We design this structure to reduce the size of network by restricting the intermediate outputs of network. Autoflow consists of multiple learnable blocks called operator splitting block (OSB) that increase the expressive power of DOSnet in modeling the operator of nonlinear PDEs. Below we describe our model in detail.
\subsubsection{Dynamical Architecture: Autoflow} \label{sec:Autoflow}
The aim of the proposed Autoflow is to obtain the solution $u(\bmx,t) \in \cX$ after a long time $T$ given the initial solution $u(\bmx,t=0)=u_0  \in \cX$, where $\cX$ is a function space of states.  
Instead of solving the Eq.\eqref{eq:PDE} directly, we employ the learnable operators $\psi_{\btheta}: \cX \rightarrow \cX$, parmaeterized by $\btheta$, to approximate the true operator in Eq.\eqref{eq:PDE}. 

The structure of Autoflow is a composition of learnable operators $\{\psi_{\btheta_i}\}_{i=1}^{M}$ as shown in Fig. \ref{fig:modelpipline}(a), where $\psi_{\btheta_i}$ is a neural network block with parameters $\btheta_i$ and $M$ denotes the number of blocks inside Autoflow. The output of Autoflow can be written as,
\begin{equation} \label{eq:Autoflow}
	\psi_{\btheta_T}\left(u_0\right) = \psi_{\btheta_M}\circ\psi_{\btheta_{M-1}}\circ\cdots \circ \psi_{\btheta_1}\left(u_0\right),
\end{equation}
The learnable parameters $\btheta = \{\btheta_i\}_{i=1}^{M}$ inside Autoflow can be optimized  via minimizing the mean square loss
\begin{equation} \label{eq:loss_func}
	\btheta^{\star} = \argmin_{
	\btheta} \cL\left(\btheta\right) = \frac{1}{K}\sum\limits_{i=1}^{K}\left(\psi_{\btheta_T}\left(u_0^{\left(i\right)}\right)-u^{\left(i\right)}\left(\bmx,T\right)\right)^2,
\end{equation}
where $K$ denotes the total number of data and $u^{\left(i\right)}$ is the $i$ th data. 

This Autoflow architecture is inspired by the existence of the evolution flow \cite{lang2012differential} for PDEs describing physical evolution of states $u\in\cX$, where $\cX$ is the function space of states. The evolution flow underlying Eq.\eqref{eq:PDE} describes a family of operators $\{\phi^t: \cX \rightarrow \cX\}$ parameterized by $t\in \bbR$ such that for any $u\in \cX$, there is
\begin{equation} \label{eq:property_flow}
	\begin{aligned}
		\phi^{0}\left(u\right)  = u_0,\quad
		\phi^s\left(\phi^t\left(u\right)\right) = \phi^{t+s}\left(u\right).
	\end{aligned}
\end{equation} 
Due to the additive group law shown in Eq.\eqref{eq:property_flow}, the true solution at time $T$, $u\left(\bmx,T\right)$ can be written in the form of 
\begin{equation} \label{eq: anlytic_solution}
			 u\left(\bmx,T\right) = \phi^{\tau_N}\circ\phi^{\tau_{N-1}}\circ\cdots\circ\phi^{\tau_1}\left(u_0\right) 
			 =\phi^{\tau_1+\tau_2+\cdots+\tau_N}\left(u_0\right) = \phi^T\left(u_0\right) 
\end{equation}
such that $\tau_1+\tau_2+\cdots+\tau_N = T$, where $N$ denotes the number of operators. In a traditional numerical scheme, $N$ is set as a large number; in other words, each operator $\phi^{\tau_i}, i=1,\cdots, N$, represents the mapping between two states over a small time step. The smaller time step allows the approximation with lower errors. However, large $N$ brings extensive computational complexity. In contrast, if $N =1$, the operator over a long time $T$ is highly nonlinear and thus is difficult to approximate directly. In this case, although a  deep neural network has powerful expressivity, it still needs enormous parameters to approximate $\phi^T$. Therefore, the desired network should have a good balance between the number of operators and the size of the networks.

The advantage of Autoflow is that it provides a balance between computationally extensive numerical schemes and large-scale deep neural networks. Specifically, it is designed to mimic the additivity of evolution flow  to reduce the  size of the neural network. One key feature of Autoflow that distinguishes it from standard DNN architectures is  that $\psi_{\btheta}$ maps $u$ into a function in the same state space $\cX$, i.e. $\psi_{\btheta}\left(u\right)$ has the same dimension as $u$. In a standard DNN, the operator inside each layer ofen maps the input function to a function in some space with different dimension (Fig.~\ref{fig:modelpipline}), and the network approximates the operator between the input state and the target state directly. While in an $M$-layer Autoflow, the output of $\psi_{\btheta_i}, i = 1,\cdots,M,$ can be interpreted as an intermediate state of the evolution process, which means that $\psi_{\btheta_i}$ approximates a particular $\phi^{\tau}$ for some $\tau$. In Autoflow, the evolution time $\tau$ is implicitly encoded by the parameter $\btheta_i$, and is learned adaptively from data. The number of operators $M$ inside AutoFlow can be much smaller than $N$ in  Eq.\eqref{eq: anlytic_solution}. Therefore, $\psi_{\btheta_i}$ learns a larger time-step underlying operator, and leads to lower computational complexity than a traditional numerical scheme with small time-step. Furthermore, the  design of intermediate output states enables  fewer parameters in the network than a standard DNN.


\subsubsection{Building blocks: Operator Splitting Blocks}
Each $\psi_{\btheta_i}$ inside Autoflow is a single operator splitting block (OSB), which contains a series of linear layers $\psi_{\cL_{\btheta_i}}$ and the special nonlinear activation functions $\psi_{\cN}$ from the original PDE; see Fig. \ref{fig:modelpipline}(a). The $i$ th OSB $\psi_{\btheta_i}$ in Autoflow is,
\begin{equation}
    \psi_{\btheta_i} = \psi_{\cN_i^n}\psi_{\cL_{\btheta_i}^n}\cdots \psi_{\cN_i^2}\psi_{\cL_{\btheta_i}^2}\psi_{\cN_i^1}\psi_{\cL_{\btheta_i}^1}
\end{equation}
where $\psi_{\cL_{\btheta_i}^{j}}, j=1,\cdots,n,$ denote  the convolutional layers with the learnable parameters, $\psi_{\cN_i^{j}} = e^{\eta_i^{j}\cN}$ are the special nonlinear activations with $\cN$ being the nonlinear operator in Eq.\eqref{eq:PDE}, $\eta_i^{j}$ are learnable scalars, $n$ is the number of convolutional layers inside one OSB. 
The last convolutional layer in $i$ th OSB $\psi_{\btheta_i}^n$ maps the features to the output with the same size as the input of this OSB.

We discuss the behavior of our network with only one OSB with one pair of linear layer and nonlinear activation (i.e., n=1). For simplicity of notations, below we omit the superscript and subscript of $\psi_{\btheta}$ and $\psi_{\cN}$.
Supposing that we use this network to approximate the evolution from initial state $u_0$ to the solution $u_{T_{1}}$,
the solution at time $T_1$ can be approximated as
\begin{equation} \label{eq:1_OSB}
    u_{T_1} \approx  \psi_{\cN}\psi_{\cL_{\btheta}}u_{0}. 
\end{equation}
Rather than the commonly used activation functions such as ReLU \cite{glorot2011deep} and Sigmoid function \cite{lecun2012efficient}, we choose $\psi_{\cN}u = e^{T_1\cN}(u)$ in OSB, which is the closed-form solution of the subequation $\frac{du}{dt} = \cN u$ with $\cN$ being the nonlinear operator from the PDE \eqref{eq:PDE}. The reason to use this PDE-based activation function is that it reflects the properties of the PDE. For example, in the Allen-Cahn equation to be discussed in Section \ref{sec:ac}, the nonlinear operator $\cN u  =  F'(u)$, where $F$ is the double  well potential that attains its global minimum value at $u =  \pm 1$. Our PDE-based activation function $\psi_{\cN}$ guarantees that the fixed points $u=\pm 1$ are kept after the activation function is applied, which will significantly accelerate the convergence of training.







However, the nature of splitting cause the splitting errors which are unavoidable by both the traditional numerical method in Eq.\eqref{eq:gcOSN} and OSB. Unlike a traditional numberical method, OSB can further reduce the splitting error by training. 
In fact, using the Baker–Campbell–Hausdorff (BCH) formula \cite{bonfiglioli2011topics},  the operator that $\psi_{\cL_{\btheta}}$ needs to approximate is,
\begin{equation} \label{eq:inverse_nl_op}
\begin{split}
       e^{-T_1\cN}e^{T_1(\cL+\cN)} 
       = e^{T_1\left(\cL+\frac{T_1}{2}[\cN, \cL]+\cO(T_1^2)\right)} 
\end{split}
\end{equation}
From Eq.\eqref{eq:inverse_nl_op}, the error is mainly induced by the Lie bracket $[\cN, \cL] = \cN\cL -\cL\cN$ for a small $T_1$.  We discuss the Lie bracket by the following three cases.
First, if $[\cN, \cL]=0$, then the splitting error vanishes, and the numerical operator splitting can obtain the exact solution with the arbitrary step size. However, for most of the PDEs, $[\cN, \cL]  \neq 0$. Secondly, when $[\cN, \cL]$ is linear, then one OSB is good enough to approximate  $e^{T_1(\cL+\cN)}$. Finally, when $[\cN, \cL]$ is nonlinear, the capacity of the linear layer inside only one OSB is limited for the approximation of $[\cN, \cL]$. Fortunately, Autoflow structure has the stacking of multiple OSBs, thus provide a nonlinear approximation for $[\cN, \cL]$. In summary, inside one OSB,  $\psi_{\cL_{\btheta}}$ is not only used to approximate the linear operator $\cL$ but also reduce the splitting error by providing a linear approximation for $[\cN, \cL]$. The nonlinearity of $[\cN, \cL]$ can be approximate by the stacking of multiple OSBs.

\subsubsection{Straightforward Adaptability}
Besides the task of approximating the solution at a specific time, our proposed network is also flexible and adaptable for some additional tasks. The design of intermediate output in Autoflow allows us to replace the default convolutional layers with  weight matrices that are consistent with the properties of the PDE.
For example,  in a NLSE, we may replace  $\psi_{\cL_{\btheta}}$ with a learnable unitary matrix \cite{jing2017tunable,kiani2022projunn} to guarantee the stability of the long-time prediction. This is validated numerically in Appendix \ref{appendix:schodinger}. Moreover, for those PDEs satisfying a certain symmetry, $\psi_{\cL_{\btheta}}$ with symmetric weight matrices can be employed to restrict the intermediate output of Autoflow close to an intermediate state on the actual solution trajectory. 

%% file: content/examples.tex
\subsection{Validation of the AutoFlow Structure} \label{sec:3toys}
We first validate the proposed Autoflow structure by considering systems without nonlinearity as benchmark problems.

\noindent\textbf{Example 1: Advection equation.}
Our first example is a standard one dimensional (1D) linear advection equation
\begin{equation} \label{advectioneq}
	\left\{
	\begin{array}{lr}
		u_t + u_x = 0, \quad x \in [-\pi,\pi],\, t>0, &  \\
		u\left(x,0\right) = \sum_{k=1}^m c_k \sin\left(kx\right)+q_k\cos\left(kx\right),&\\
		u\left(\pi,t\right) = u\left(-\pi,t\right), &  
	\end{array}
	\right.
\end{equation}
where $c_k, q_k \sim \mathcal{N}(0,\,1)$ with $\mathcal{N}(0,\,1)$ being normal distribution and $m=10$. 
We discretize the spatial domain $ [-\pi,\pi]$ into 200 uniform grid points. In order to predict the behavior of the solution $u$ from 0s to 0.3s, we generate 5000 data pairs of $u\left(x,0\right)$ and $u\left(x,0.3\right)$ from the analytical solution as the input and ground truth of our network, respectively. These 5000 data are divided into training data of 3750 and test data of 1250.

\noindent
\textbf{Example 2: Diffusion equation.}
The second example is a 1D diffusion equation
\begin{align} \label{eq:diffusion}
    u_t = u_{xx}, \quad x \in [-\pi,\pi], \, t>0,
\end{align}
with the same initial and boundary conditions as those in Eq.\eqref{advectioneq}. 
In this example, we examine the evolution of this system from 0s to 0.03s. The generation of data is the same as that in  Example 1.

In all two examples, we use a 3-block linear Autoflow. Our model is trained using the Adam optimizer through $100$ epochs with a learning rate of $10^{-3}$ and $L_2$ regularization of $10^{-4}$. Inside Autoflow, each block contains one convolutional operator of kernel size of $21$. Constant initialization of weights $\frac{1}{k}$ is used in our model.

Results obtained using a 3-block linear DOSnet on these two linear examples are shown in Figs.~\ref{fig:toy_acc} and \ref{fig:3toysresult}. We observe that DOSnet has a solid capacity to approximate the solution of the PDEs at a specific time from the results in Fig.~\ref{fig:toy_acc}. Besides, one key feature that distinguishes DOSnet from the standard DNNs is that the AutoFlow structure allows the linear DOSnet to perform as a numerical operator with a fixed step size.
From the results in the left column of Fig.~\ref{fig:3toysresult}, we observe that the intermediate outputs of DOSnet characterize the transition states at the referenced time levels in the original PDEs. The referenced time of the exact solutions is obtained by calculating the $L_2$ losses between the intermediate outputs of the DOSnet and the exact solutions of the original PDE in the whole time trajectory and choosing the corresponding time of the exact solution with the minimum losses. Moreover, it shows that the referenced time matched by the intermediate outputs of DOSnet divides the whole time trajectory equally. Also, the profiles of the weights in the third column indicate that the weights of different blocks in DOSnet converge to be the same after training. Therefore, the same weight inside each block of DOSnet and the intermediate outputs with equal time steps demonstrate that the linear DOSnet performs as a numerical operator with a fixed step size. We provide a theoretical analysis for these behaviors of weights in Appendix \ref{appendix:proof}. 


\begin{figure}
	\centering
	\includegraphics[width=1\textwidth]{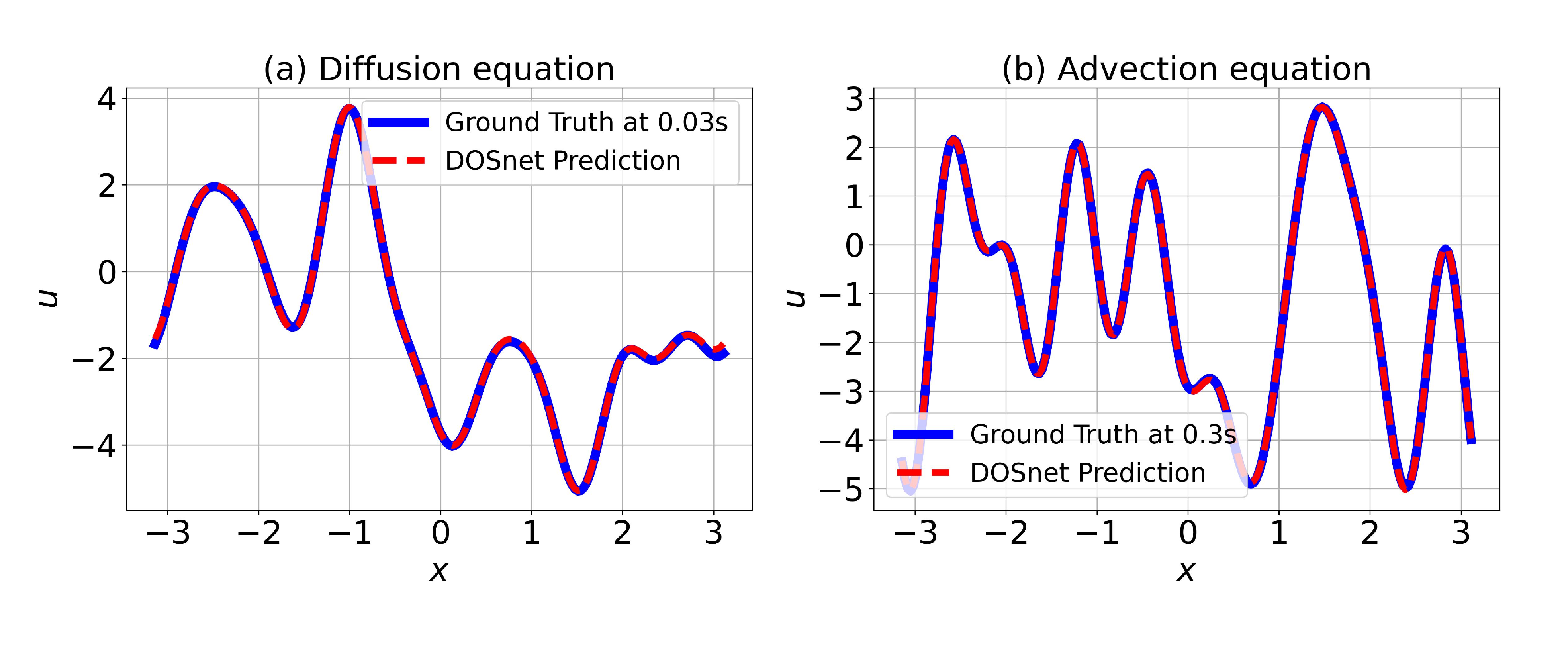}
	\caption{Predictions of a 3-block DOSnet (red dashed line)  on (a) diffusion equation, (b) advection equation compared with their ground truths (blue solid line) at the target time.}\label{fig:toy_acc}
 \end{figure}
 
\begin{figure}
	\centering
	\includegraphics[width=1\textwidth]{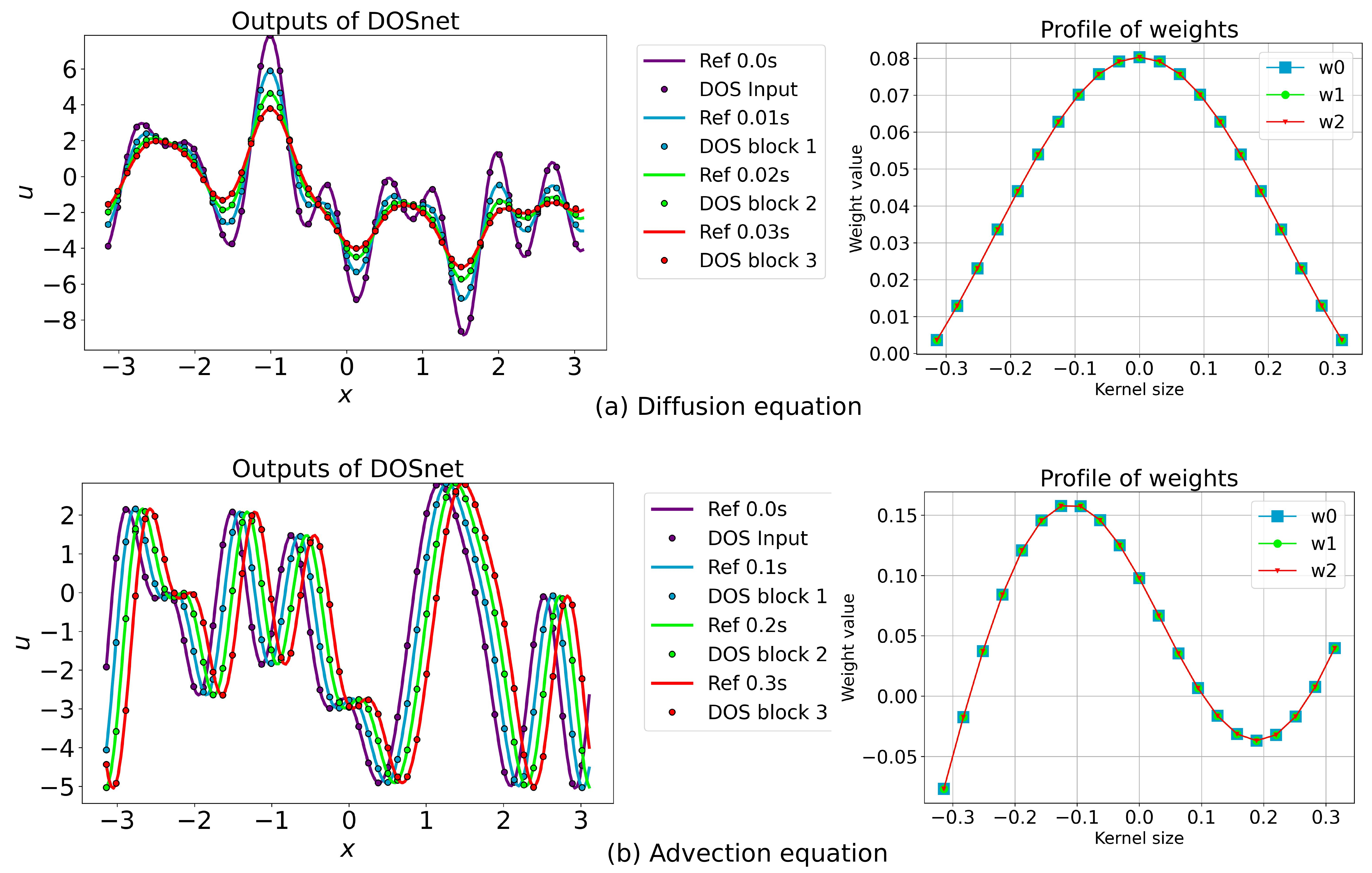}
	\caption{Intermediate outputs and weight profiles of a 3-block DOSnet on (a) diffusion equation, (b) advection equation. 
 In the left column, comparisons between outputs from the intermediate blocks of DOSnet (dashed line) and the exact solutions at the referenced time (solid line) are visualized. In the right column, the converged weights of each block in DOSnet after training are visualized.}
	\label{fig:3toysresult}
\end{figure}

%% file: content/case_study.tex
\subsection{Application: Allen-Cahn equation} \label{sec:ac}
In this subsection, we apply DOSnet to the Allen-Cahn equation, which is a reaction-diffusion equation that describes the process of phase separation and has been widely used as a phase field model for the interface dynamics \cite{du2020phase}. We consider a two dimensional Allen-Cahn equation with periodic boundary condition,
	\begin{equation}
		\label{eqn:allen_eq}
		\begin{split}
			&\partial_t u\left(\bmx,t\right) - \epsilon^2 \nabla^2 u\left(\bmx,t\right) + [u\left(\bmx,t\right)]^3-u\left(\bmx,t\right) = 0 , \quad  \bmx\in [-1,1]^2,\,t>0, \\
&u\left(\bmx,0\right) = u_0\left(\bmx\right), \quad \bmx \in [-1,1]^2.
		\end{split}
	\end{equation}
In our calculations, the initial data $u_0$ is given by a Fourier series whose frequencies are less than 8 with random coefficients sampled from the Gaussian distribution $\mathcal{N}\left(0,I\right)$ , 
and $\epsilon$ is a parameters which indicates the width of region between phase separation. The obtained results using our DOSnet will be compared with those obtained by using the numerical method of the symmetrized Strang splitting scheme \cite{strang1968construction}. See Appendix \ref{app:numerical_ac} for the formulation of this numerical scheme being used. The numerically "exact" solutions are obtained by using this splitting scheme using a fine mesh and small time step. 



We first examine the effect of model structure on accuracy. The accuracy of a neural network depends on are two factors: (1) network architecture (DNN or Autoflow), (2) nonlinear activation (ReLU or OSB), out of which Autoflow and OSB are used in our DOSnet.
Table \ref{table: test and infer error}  shows the test errors at time $T$ and $2T$ under four configurations that are combinations of the above two factors. The results show that the configuration, Autoflow+OSB, achieves the lowest errors in predicting the solutions at time $T$ and $2T$. 

We further conduct the $2^2$ full factorial design to analyze the results quantitatively and quantify how different model settings affect the network performance. Note that the full factorial design is an experiment used to evaluate the effect of factors and interaction between factors on the response variable. The factors and corresponding levels are listed in Table \ref{table:factorial design}. We assume that the relation between the test error $y$ and each factors can be described by a regression model: $\log(y) = q_0 + q_Ax_A + q_Bx_B + q_{AB}x_{AB}$, where $q$ is the effect of the corresponding factor, $x_A, x_B, x_{AB}$ are the levels of factors A (network architecture), B (nonlinear activation), and AB (interaction), respectively, as  listed in Table \ref{table:factorial design}. The obtained effects  at time $T$ and $2T$ are shown in Table \ref{table:test Effect Estimate}. As the results shown in in Table \ref{table:test Effect Estimate}, for the prediction at time $T$, the interaction factor $AB$ contributes the largest effect $q_{AB} = -1.2480$ to the error $\log(y)$. That means the configuration with the level $x_{AB} = +1$ should be chosen to reduce the error. Therefore, it is necessary to use the combination Autoflow+OSB/ DNN+ReLU in order to give the lowest errors in the prediction at time $T$. Similarly, for the inference at time $2T$, the nonlinearity with the factor $B$ has the most significant effects $q_B = -1.7806$. Therefore, the configuration with level $q_B = +1$, i.e., OSB, is the best choice for the prediction in time $2T$. By these two conclusions, the Autoflow+OSB combination adopted in our DOSnet is the best configuration.
\begin{table}[h]
	\renewcommand\arraystretch{1.5}
	\centering
	\scalebox{1}{
		\begin{tabular}{ccccccccc} %
			\hline\hline
			\multicolumn{3}{c}{\multirow{2}*{Method}}& \multicolumn{2}{c}{Time $T$}& &\multicolumn{2}{c}{Time $2T$}\\
			\cmidrule{4-8}
			\multicolumn{3}{c}{}&DNN&Autoflow&&DNN&Autoflow&\\  %
			\hline %
			\multicolumn{3}{c}{ReLU}&0.01024& 0.1875 & & 1.4330& 0.41988& \\   %
			\multicolumn{3}{c}{OSB}&0.01825&\textbf{0.00226} &&0.08793&\textbf{0.00552}&  \\
			
			\bottomrule %
	\end{tabular}}
	\caption{Relative errors of four model settings at time $T$ and time $2T$.}
	\label{table: test and infer error}
\end{table}

\begin{table}[h]
\renewcommand\arraystretch{1.5}
\centering
\scalebox{1}{
    \begin{tabular}{ccc}
        \hline\hline
        Factor/ Level & +1 & -1 \\
        \midrule
        A: Network & Autoflow & DNN \\
        B: Nonlinearity & OSB & ReLU \\
        AB: Interaction & Autoflow+OSB/DNN+RELU
        
        & Autoflow+RELU/DNN+OSB 
        
        \\
        \bottomrule
\end{tabular}}
\caption{$2^2$ Full factorial design.}
\label{table:factorial design}	
\end{table}

\begin{table}[h]
\renewcommand\arraystretch{1.5}
\centering
\scalebox{1}{
\begin{tabular}{ccc}
    \hline\hline
    {Factor/ $q$}& Time $T$ & Time $2T$ \\
    \midrule
    A: Network & 0.2046 & -0.9989 \\
    B: Nonlinearity & -0.9601 & \textbf{-1.7806} \\
    AB: Interaction &\textbf{-1.2490}
    & -0.3851	\\
    \bottomrule
\end{tabular}}
\caption{Effect estimate in the $2^2$ full factorial design.}
\label{table:test Effect Estimate}
\end{table}

Next, we explore generalization of the four configurations in a longer-term evolution. As shown by the pipeline in Fig.\ref{fig:pip-long-evol-main-text}, our model is trained from $u_0$ to $u_T$, and then this learnt model is used to iteratively inference the evolution for time steps, i.e., $u_T$ to $u_{2T}$,  $u_{2T}$ to $u_{3T}$, ......, and $u_{7T}$ to $u_{8T}$, where $T = 5s$. Then, we record the inference errors at $u_T, u_{2T},\cdots,u_{8T}$.  The inference errors at next 8 time steps after the trained one are shown in Fig.\ref{fig:4comp-long-error}, and their visualization results in Fig.\ref{fig:vis-long-generalization-main-text}. From these results, we can observe that DOSnet always achieves the best performance in predicting the true evolution in Allen-Cahn model from $T$ to $8T$. The error of the standard setting DNN+RELU exhibits exponential increasing in the long-term inference (Figure \ref{fig:4comp-long-error} (4)). While if any one factor in model is changed (Figure \ref{fig:4comp-long-error} (1,2,3), the curves of inference errors are linear, with a large drop in error, among which the Autoflow+OSB combination achieves the lowest error $0.022$ even at time $8T$. 
\begin{figure}[h]
    \centering
    \includegraphics[width=0.8\linewidth]{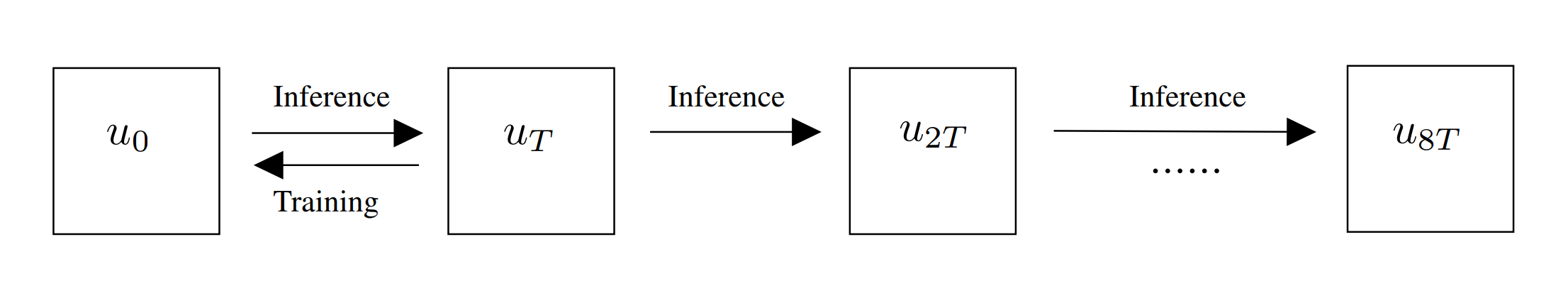}
    \caption{Inference pipeline of $8 T$ evolution, $T=5s$.}
    \label{fig:pip-long-evol-main-text}
\end{figure}

\begin{figure}
	\centering
	\includegraphics[width=1\linewidth]{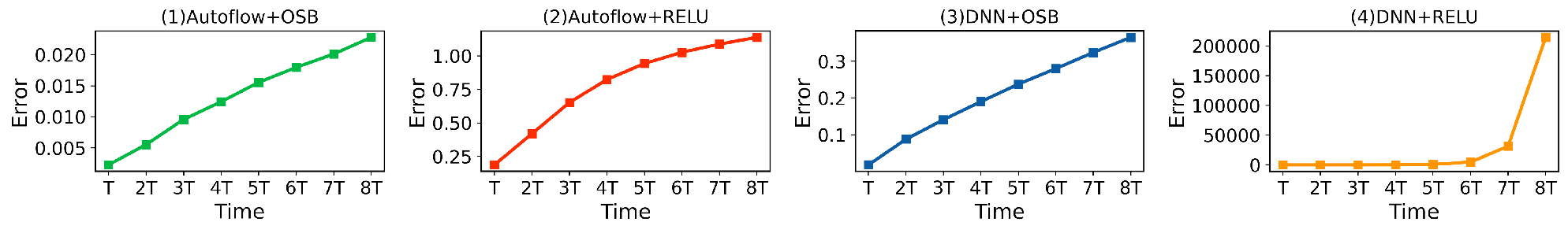}
	\caption{Relative errors at 8 time steps after the training time, with 4 different model settings.}
	\label{fig:4comp-long-error}
\end{figure}

\begin{figure}
	\centering
	\includegraphics[width=1\linewidth]{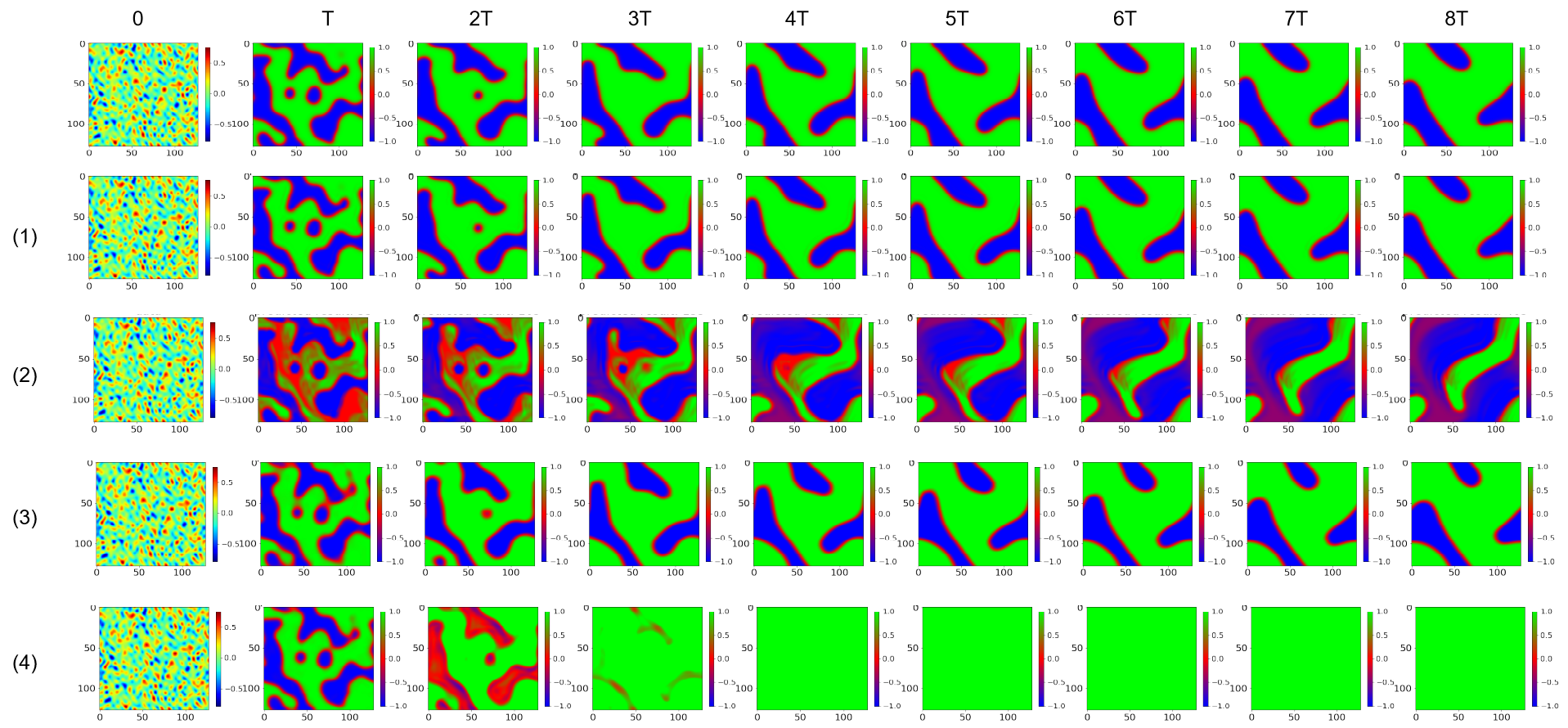}
	\caption{Visualization results of long-time evolution (from 0 to 8T)  with 4 model settings, where T is the training time. The first row is the ground truth generated by a numerical method  (see Appendix \ref{app:numerical_ac}) from states 0 (the initial one) to 8T. The other rows show long-time prediction by (1) DOSnet (contains Autoflow and OSB). (2) Autoflow + ReLU. (3) DNN+OSB. (4) DNN+ReLU.}
	\label{fig:vis-long-generalization-main-text}
\end{figure}

We also investigate the time trajectory of the evolving system using our model. Since Autoflow naturally guarantees that the intermediate states are interpretable, we would like to examine if these intermediate outputs present true states of the Allen-Cahn evolution, and if they are, what are the time steps between them? Thus, we match these intermediate outputs $O_n, n = 0,1,\cdots, N$, where N is the number of blocks , within a large amount of numerical results $y\left(t_n\right), 0 \leq t_n \leq T,$ by minimizing the errors between them. That is, the predicted time $t_n$ for each intermediate output is $t^\star_n = \argmin_{t_j} \Vert O_n - y\left(t_j\right)\Vert_2^2, j = 1,\cdots,K$, where $K$ is the number of numerical results. The matched results for models with 2-7 blocks obtained by 5 trials are presented in Fig.\ref{fig:time-matching}. We find that, for the models with few blocks ( 2-5 in Fig.\ref{fig:time-matching} (a-d)), the time steps between their intermediate outputs are nearly constant and divide the total evolution time equally. While with number of blocks increasing ( 6 and 7 in Fig.\ref{fig:time-matching} (e-f)), the total evolution time cannot be divided into shorter time steps. We believe that this behavior different from numeral solutions may be due to the difficulty of training in DNN-based model with deeper layers. The visualization of intermediate outputs of 5-block DOSnet are shown in Fig.\ref{fig:intermediate_output}. Visualization of intermediate outputs of models with other number of blocks are shown in the supplementary materials.

\begin{figure}
	\centering
	\includegraphics[width=0.8\linewidth]{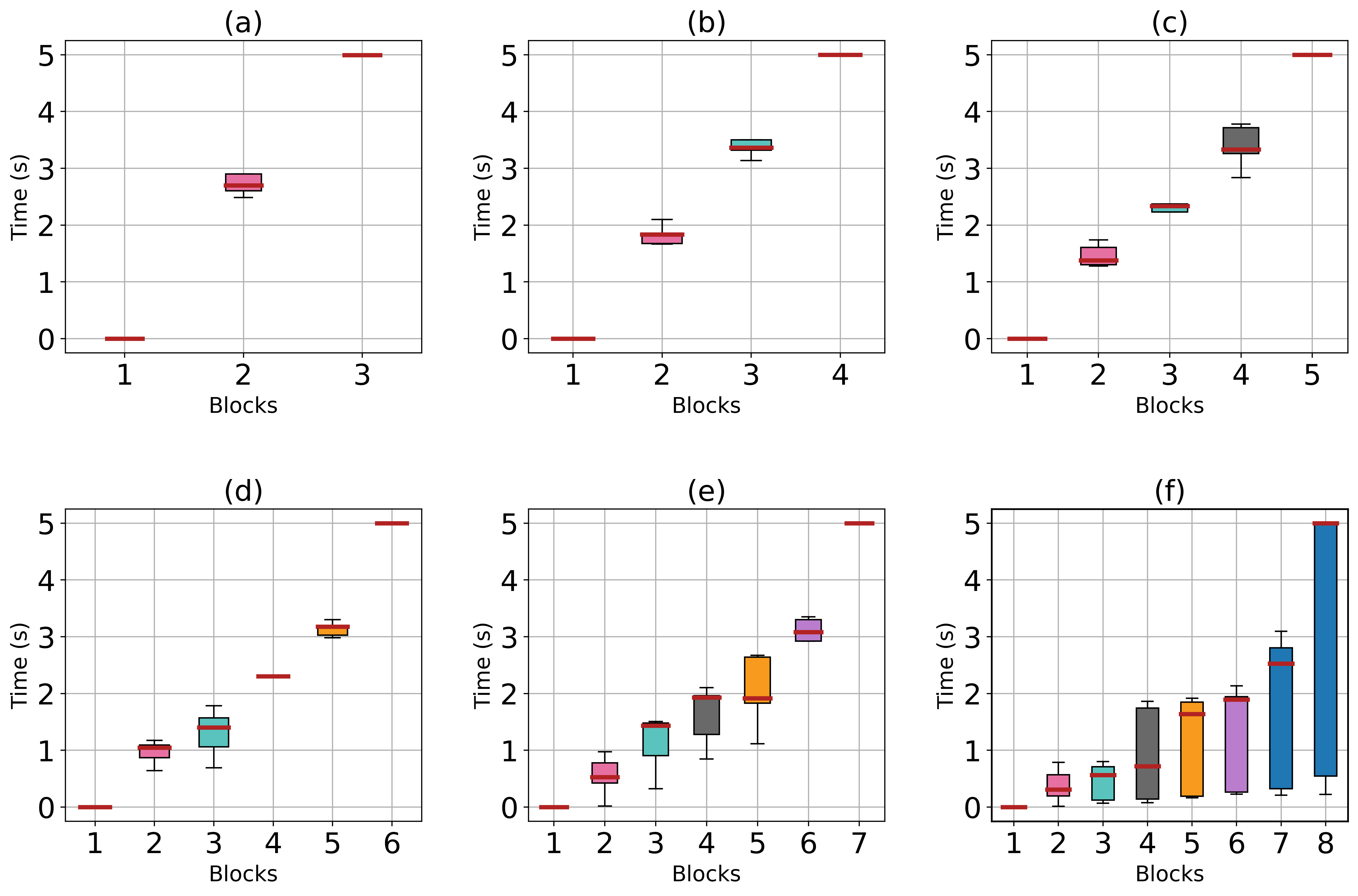}
	\caption{Boxplot results of time-step matching using 2-7 blocks models by 5 trials. The bars show the interquartile range and the red lines denote the medians of 5 trial results.}
	\label{fig:time-matching}
\end{figure}

\begin{figure}
    	\includegraphics[width=\linewidth]{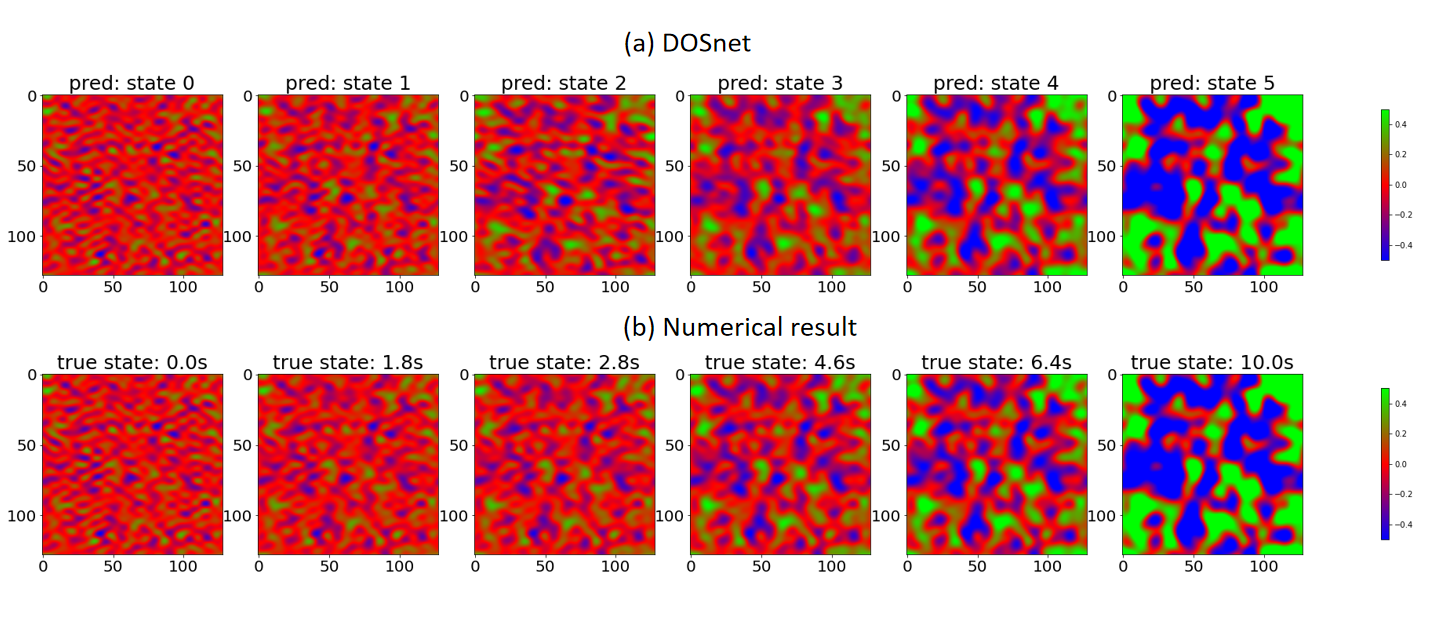}
    	\caption{(a) Visualization of intermediate outputs of 5-block DOSnet. (b) The numerical results at the matched time.}
    	\label{fig:intermediate_output}
\end{figure} 
We also examine the computational complexity, we compared the trainable parameters between 5-block DOSnet and 5-block baseline DNN used in previous experments. Under the same kernel size of each convolution, our DOSnet has 19360 trainable parameters, which is much fewer than 96896 parameters of baseline DNN. This shows that the Autoflow structure in DOSnet is able to significantly reduce the computational complexity. Note that using a traditional numerical scheme, in order to reduce the error, the time step has to be very small, leading to a large number of steps. Once it is trained, our network DOSnet is able to describe the dynamics of Allen-Cahn equation with much lower computational complexity.

Finally, we discuss the effect of number of blocks in DOSnet. The default number of the blocks (OSBs) in our model is set to be 5. This setting is empirically the best. Fig.\ref{fig:all1layers} shows the logarithmic relative errors by 5 trials for solving the  Allen-Cahn equation for different number of OSBs. Each OSB has the same structure: two convolutions and two activation functions, and each convolution has 16 channels with kernel size of 11. We observed that the 5-block model achieved smallest average relative error with acceptable variance.
\begin{figure}[H]
	\centering
	\includegraphics[width=0.5\linewidth]{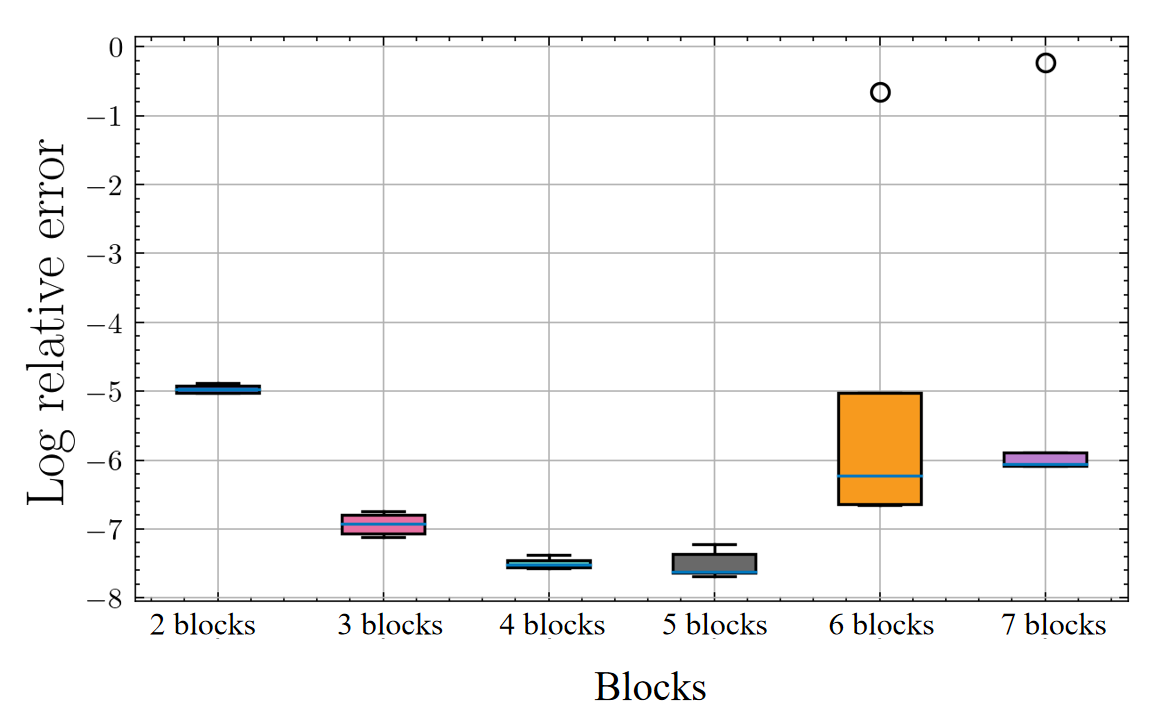}
	\caption{ Logarithmic relative errors for models with 2-7 blocks by 5 trails for solving the Allen-Cahn equation.}
	\label{fig:all1layers}
\end{figure}

\subsection{Application: Nonlinear Sch\"{o}rdinger equation}\label{sec:nlse}
In this section, we aim to address the nonlinear compensation problem in optical communication by applying DOSnet to solve the NLSE inversely. 
Since being invented by Sir Charles K. Kao in the 1960s, the technology of fiber-optic communication has been rapidly developed and has become one of the foundations in the era of information \cite{kao_dielectric-fibre_1966}. One of important applications in the signal processing for the modern optical fiber transmission systems is nonlinear compensation, in which nonlinear Schr\"odinger equations (NLSE) are solved inversely to recover the original signals at the transmitter from the highly distorted signal at the receiver. The distorted signal is caused by four factors during transmission: noise from inline optical amplifier, attenuation, dispersion, and nonlinear effect from the characteristics of fiber which can be described by NLSE. 
In order to recovery the original information of this signal, nonlinear compensation is required. More on the optical transmission system are given in Appendix \ref{app:nlse}.

Mathematically, the forward propagation of a signal in the fiber can be described by the NLSE \cite{agrawal_fiber-optic_2010},
\begin{align}\label{eq:NLSE}
	\frac{\partial u}{\partial x} = -\frac{\alpha}{2} u - \frac{i\beta}{2}\frac{\partial^{2}u}{\partial t^{2}} + i\gamma \lvert u \rvert^{2}u,
\end{align}
where the propagation parameters $\alpha>0, \beta\in\bbR, \gamma\in\bbR$ are characteristics of the optical fiber and $i=\sqrt{-1}$. Here $u$ is a complex-valued function of $(t,x)\in\bbR\times[0,X]$, which represents the light field in the fiber whose length is $X$. The simulated distorted signal at the receiver can be obtained by solving  Eq.\eqref{eq:NLSE} with a given initial condition $u(t,x=0)=u_{0}(t)$, which is the signal to be transmitted to a given distance $x$.
Notice that in this equation, the optical signal is a function of $t$ and it propagates along $x$, which indicates the distance of propagation.
Conversely, the backward propagation, which recoveries the clean signal at the transmitter from the distorted signal at the receiver, can be simulated by solving the NLSE \eqref{eq:NLSE} inversely, i.e., taking the negative sign of the propagation parameters $\alpha, \beta$, and $\gamma$ in Eq.\eqref{eq:NLSE}. Traditional numerical methods \cite{Ip2008CompensationOD,Weideman1986SplitstepMF} for the simulation of backward propagation is based on the operator splitting of  linear and nonlinear parts as two sub equations:
\begin{align} 
	&\frac{\partial u}{\partial x} = \cL u = \frac{\alpha}{2} u + \frac{i\beta}{2}\frac{\partial^{2}u}{\partial t^{2}}, \label{eq:NLSE_linear_ops}\\
	&\frac{\partial u}{\partial x} = \cN u = -i\gamma \lvert u\rvert^{2}u. \label{eq:NLSE_nonlinear_ops}
\end{align}
Note that the solution of the linear PDE \eqref{eq:NLSE_linear_ops} is
$u(t,x+\xi) = \sqrt{\frac{i}{2\pi\beta\xi}}\exp{\left(\frac{\alpha\xi}{2}+\frac{i t^{2}}{2\beta\xi}\right)} * u(t,x) $, where $*$ denotes the convolution, and $\xi$ is the distance of propagation. It describes the attenuation and dispersion of the signal along the distance of propagation $\xi$. The solution of the nonlinear PDE \eqref{eq:NLSE_nonlinear_ops} is $u(t,x+\xi) = u(t,x)\exp{(-i\gamma\xi\lvert u(t,x)\rvert^{2})}$, which gives the nonlinear change of phase rotation in the signal.

In DOSnet, we use a complex-value convolution with learnable parameters instead of the analytic solution for the linear PDE \eqref{eq:NLSE_linear_ops}. The nonlinear activation function in DOSnet follows the solution of the nonlinear PDE \eqref{eq:NLSE_nonlinear_ops} rather than the common-used activation function in the neural network, such as ReLU, sigmoid.


We examine the performance of our algorithm in nonlinear compensation for optical signal. The performance of the recovery is measured by the bit-error-rate (BER), which is the
ratio of error bits to the total number of transmitted bits at the decision point,
and is a commonly used measurement in optical transmission system.
Smaller BER indicates a better algorithm. The data we used in experiments are shown in section \ref{sec:method_exp_set} and our model setting is described in \ref{sec:model_set}.

In our experiments, we compare the performance of nonlinear compensation with the classical numerical algorithm, SSFM, under the launched power from $-6$ dBm  to $6$ dBm. Larger launched power indicates stronger nonlinear effect during transmission. When the launch power is low, the nonlinear effect in NLSE \eqref{eq:NLSE} is small, the difficulty of compensation mainly comes from the noise of inline amplifier. BER results obtained using DOSnet and comparisons with those of SSFM are shown in Fig.\ref{fig:nlse_BER}. From the results, we can observe that DOSnet achieves much lower BER than SSFM when the launched powers are large (0-6dBm), i.e., the nonlinear effect is strong. This demonstrates that our method is better at  handling the nonlinear  interaction  effect than traditional numerical methods.  Note that when the launched power is low, DOSnet gives slightly higher BER than SSFM. This is because the linear operators in DOSnet is learned from noisy data, while the results of SSFM are obtained by the deterministic equation \eqref{eq:NLSE_linear_ops} without consider the noise. Thus the determined operators of SSFM is more robust to noise.

\begin{figure}[h]
\centering
    	\includegraphics[width=0.8\linewidth]{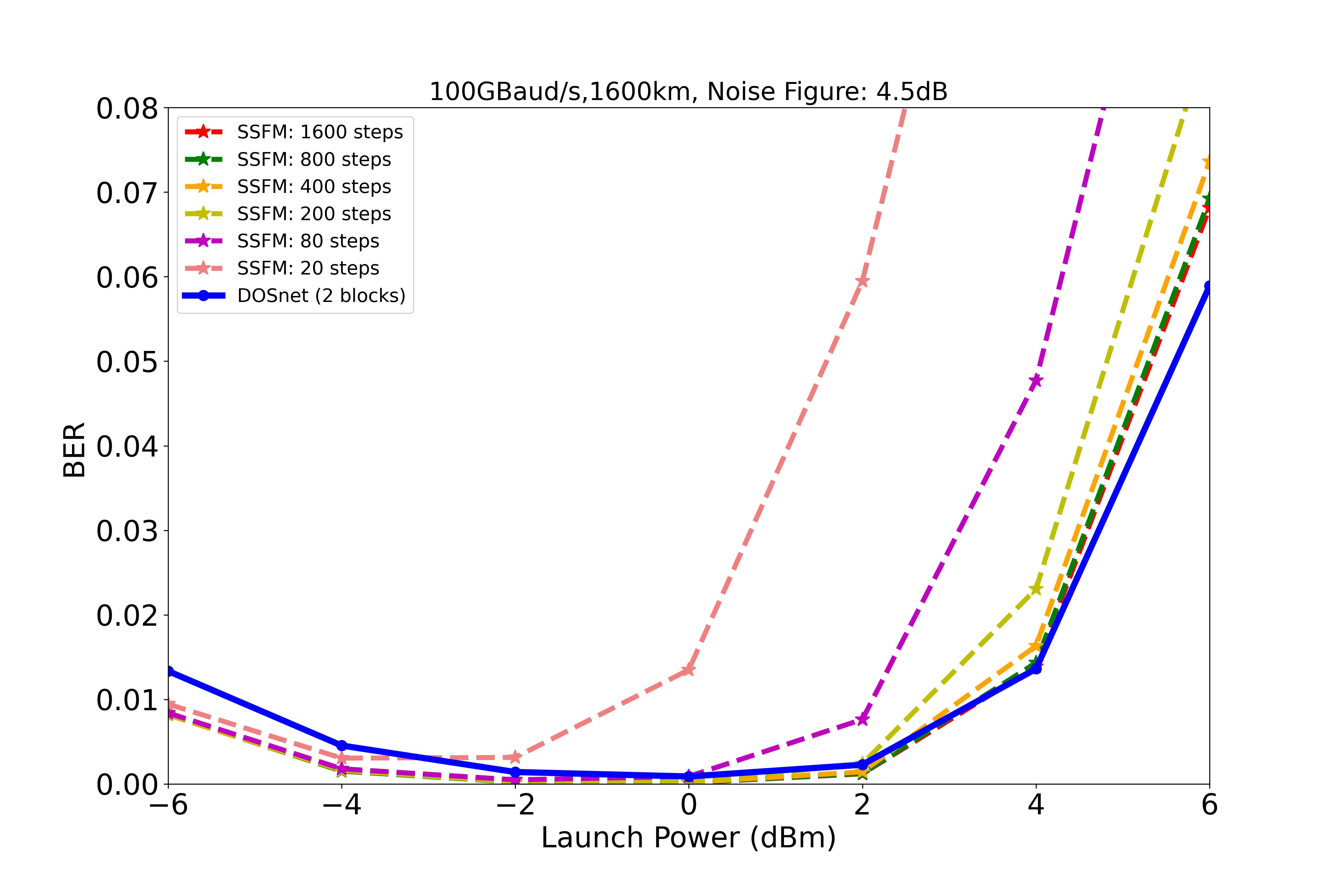}
    	\caption{ BER (Bit error rate) of the results of DOSnet and multi steps SSFM algorithm for launch power from $-6$ dBm to $6$ dBm.}
    	\label{fig:nlse_BER}
\end{figure} 

We also examine the computational complexity in the experiments. We compare the results of 2-block DOSnet (24,008 parameters) with multiple steps SSFM, and the results in Fig.\ref{fig:nlse_BER}. It can be seen from the results that our method achieves the same BER with as the 80 steps SSFM in 0dBm. As the launch power increases, SSFM requires more numerical steps to obtain the same performance as DOSnet. With more numerical steps, the performance of SSFM converges, and when the launch power is below 6dBm, its converged BER is still much higher than that of DOSnet. Therefore, DOSnet shows much better BER and lower computational complexity than SSFM from 0dBm to 6dBm. In Fig.\ref{fig:dosnet_result_vis}, we show some qualitative results of the signals recovered by DOSnet. 

\begin{figure}[h]
\centering
    	\includegraphics[width=1\linewidth]{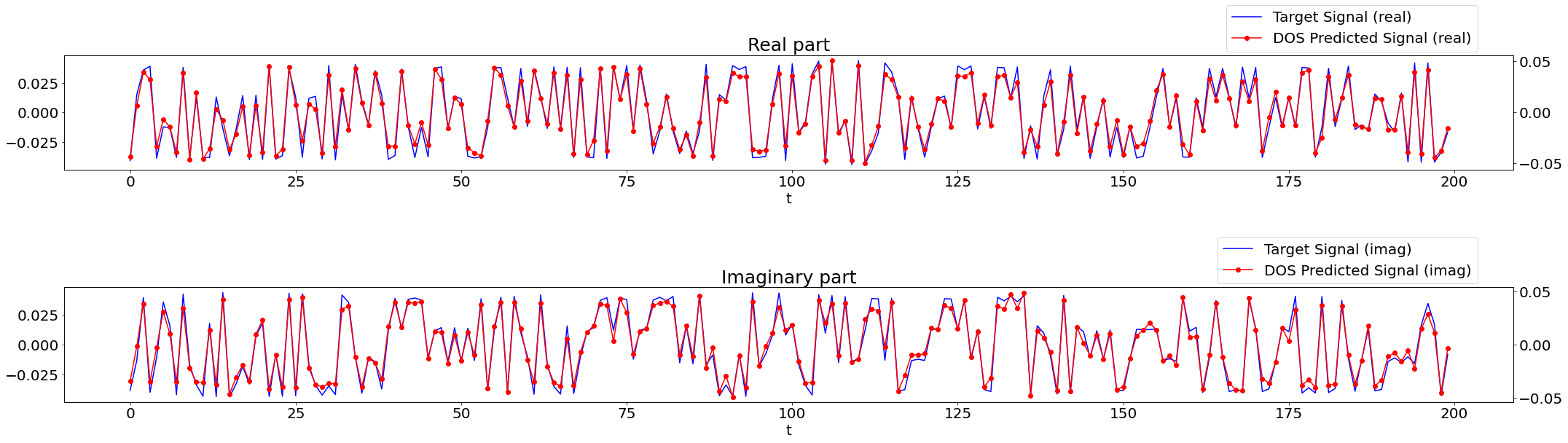}
    	\caption{Signals recovered by DOSnet and the original signals at transmitter (i.e., the ground truth) at $2$ dBm.}
    	\label{fig:dosnet_result_vis}
\end{figure} 

Finally, we compare the performance of the common-used models in deep learning, such as multilayer perceptron (MLP) \cite{hastie2009elements}, VGG16 \cite{simonyan2014very}, ResNet\cite{he_deep_2016}, and bidirectional long short-term memory network (Bi-LSTM) \cite{graves2005bidirectional} with that of DOSnet, and the results are summarized in Table. \ref{table:nlse_nn}. The results show that DOSnet can achieve the best BER with the fewest model parameters. We observe that DOSnet outperforms MLP, a basic neural network, with 62.5 times fewer parameters. Moreover, from the result of VGG16 and ResNet18, the standard DNNs with high-dimensional features, we observe that their huge amounts of parameters do not benefit the performance of solving the inversed NLSE. Also, the performance of Bi-LSTM, one of the classical recurrent neural networks, is worse than that of DOSnet, even with much more parameters. 

\begin{table}[h]
\renewcommand\arraystretch{1.5}
\centering
\scalebox{1}{
    \begin{tabular}{ccc}
        \hline\hline
        Models & Parameters & BER \\
        \midrule
        MLP \cite{hastie2009elements} & 1.5M & 0.0029\\
        VGG16 \cite{simonyan2014very} & 29.9M &0.0746 \\
        ResNet18 \cite{he_deep_2016} & 7M &0.1858\\
        Bi-LSTM \cite{graves2005bidirectional}&48.1M& 0.0086\\
        \textbf{DOSnet (Ours)} & \textbf{0.024M} & \textbf{0.0009}\\
        \bottomrule
\end{tabular}}
\caption{Results of DOSnet comparing with baseline DNNs based on the optical transmission system described in Section \ref{sec:method_exp_set_nlse} with $0$ dBm launch power. We trained those baseline DNNs using the same training data as those of DOSnet. M denotes a million.}
\label{table:nlse_nn}	
\end{table}

%% file: content/method_NCS.tex
\subsection{Model setting} \label{sec:model_set}
\textbf{Architecture:} For the Allen-Cahn equation, the architecture of Autoflow contains five operator splitting blocks (OSBs) by default. Each OSB consists of two convolutional layers (hereinafter referred to as \textit{conv}) and two activation layers (referred to as \textit{act}), i.e. \textit{conv-act-conv-act} with channel number alternating from $1-16-1$. Each convolutional operator has kernel size of $11$. The baseline DNN for solving the Allen-Cahn equation has layers: \textit{conv-act-[conv-act]*3-ac-conv}, where each convolution layer converts the input into 16 channels with kernel size of 11, expect for the last convolution which maps the input into the output of one channel with kernel size of 11. For the NLSE, we choose a two-block DOSnet, where each block contains one convolution and one activation. Each convolution has 2 channels with kernel size of 3001. 

\noindent\textbf{Optimizer:} 
Adam optimizer is used for training with learning rate 0.0004 in the Allen-Cahn equation and 0.001 in the NLSE. The learning rate decays to its 10\% for every 30 epochs. The batch size is 64 without use batch normalization. The weight initialization in networks are chosen as orthogonal initialization in NLSE and as Kaiming uniform initialization in Allen-Cahn equation.

\noindent\textbf{Network training:} All trainings in this paper are based on the deep learning framework Pytorch \cite{paszke2019pytorch} and all computation are implemented on a machine equipped with an Intel Core i9-10900X CPU and a NVIDIA GeForce RTX 2080 Super GPU.

\subsection{Experimental set-up}\label{sec:method_exp_set}
\subsubsection{Allen-Cahn Equation}
In the Allen-Cahn equation, we generate a dataset of 500 data pairs by Strang splitting \cite{strang1968construction} with time step $2\times 10^{-4}$s. The hyperparameter $\epsilon$ in Eq.\eqref{eqn:allen_eq} is 0.02. Each data pair contains $u_0\left(\bmx\right)$ and $u\left(\bmx,T\right)$, where $T=5s$ for $\bmx\in[-1,1]^2$ and it is discretized into $128\times128$ uniform grids.
\subsubsection{NLSE}\label{sec:method_exp_set_nlse}
The data  we used in NLSE experiment is simulated signals in optical fiber of 1600 km, with transmission rate 100 GBaud/s. A sequence of 16-QAM signal contains 96,000 symbols with sampling rate of 4 samples per symbol (sps) is generated at the transmitter. Each symbol is a segment of the signal with the same length. The signal is transmitted through the optical fiber with 20 spans. The parameters in NLSE are $\alpha=0.063, \beta=-21.68, \gamma=1.66$ for the fiber, which is solved by the split-step Fourier method (SSFM) \cite{Weideman1986SplitstepMF, Hardin1973ApplicationOT} to obtain the training data. Noise with 4.5 dB amplifier noise figure is added during the transmission every 80km (see Appendix \ref{app:nlse} for details). 

Since the goal is to recover the original signal from the distorted signal, the distorted data at the receiver is the input of our network and the undistorted data at the transmitter is the target for DOSnet is to recover. Thus, the data for training DOSnet are two signals with length of 384,000 samples each. In practice, each signal is first downsampled to 2 sps and then cropped into {24,000} segments with 6016 samples with stride of 8 samples. 
Inside each data of 6016 samples, the middle 16 samples are valid samples and the remaining 3000 samples on both side of them are used for padding. After passing the 2-block DOSnet where each convolutional layer with kernel size of 3001, the output of DOSnet is a predicted recovery data with 16 samples (See Fig.\ref{fig:datanlse}). The training and validation set contain 12,000 segments, respectively. 
\begin{figure}[h]
	\centering
	\includegraphics[width=\linewidth]{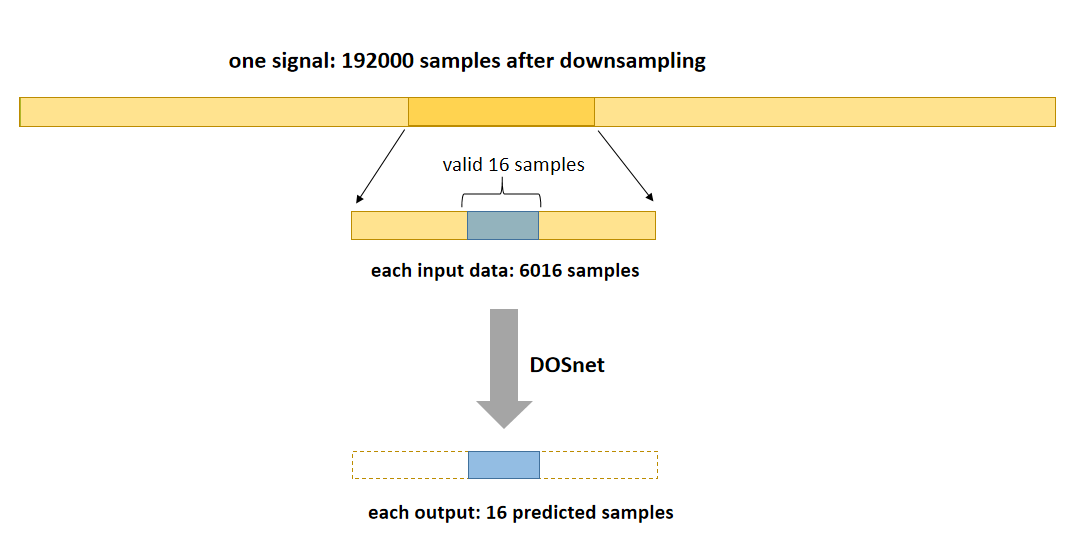}
	\caption{Singal preprocessing in NLSE example. A long signal after downsampling at receiver contains 192000 samples. This signal is cropped into 24000 segments. Each segment has the length of 6016 samples (16 samples are valid, the remaining samples are for padding). Those segments are the input of DOSnet. After passing through DOSnet, only 16 samples are predicted.}
	\label{fig:datanlse}
\end{figure}



%% file: content/discussions.tex
This paper proposes the DOSnet, an operator-splitting-based neural network, to solve evolutionary PDEs. The structure of the DOSnet contains two levels: on a coarse level, it is a general dynamical architecture Autoflow (autonomous flow) describing the action of the evolution operators, which consists of multiple learnable OSB (operator splitting blocks); and inside each OSB, there is a series of linear layers and particular nonlinear layers based on the nonlinear operator in the PDE to reflect its properties. A critical 
property of DOSnet is that the input and output of each block have the same dimension to mimic the evolution of the PDE. This feature distinguishes DOSnet from the standard DNNs: each block of DOSnet models the transition between the states in the physical evolution systems rather than mapping the input into some higher dimensional features. Moreover, the design of identical input and output dimensions in DOSnet vastly reduces the parameters of the neural network. It makes DOSnet a lightweight neural network that meets the need for efficient simulation in the  engineered applications, such as the nonlinear compensation in optical fiber transmission systerms. 

We validate DOSnet on several types of operator-decomposable differential equations. For solving the Allen-Cahn equation, experimental results show that DOSnet achieves better accuracy in predicting the solutions than the baseline DNN. DOSnet also shows several unique properties that the baseline DNN cannot represent. First, DOSnet can `stabilize' the error increasing in the long-term evolution. The curve of the generalization errors of the baseline DNN in long-term evolution grows exponentially, while that of the DOSnet represents linearity with a significant drop in errors. Moreover, we observe that intermediate outputs of DOSnet can match the time trajectory of the evolving system, and their time steps are nearly constant. 

We also apply DOSnet to solve the inversed NLSE, which has important applications in the nonlinear compensation of optical communication. In this application, DOSnet has demonstrated its significant advantages in mitigating nonlinear distortion with low computation cost. Compared to the traditional numerical schemes, the SSFM with multiple iteration steps, a 2-block DOSnet can achieve better BER in high launched powers.
Compared to the standard DNNs with millions of parameters, our 2-block DOSnet has a significantly better BER than theirs with at least $62.5$ times fewer parameters. 

The proposed DOSnet is developed based on operator splitting of the PDEs. The results and analysis in this paper also demonstrate that it is much more efficient to utilize a neural network with the help of the physical properties of the underlying PDEs rather than extracting the general high-level features with millions of parameters.

%% file: content/appendix.tex
\setcounter{figure}{0} 
\setcounter{table}{0}
\section{Error Analysis for Operator Splitting Block (OSB)}

We show error analysis for OSB in details based on the Allen-Cahn equation in main text section \ref{sec:ac}. The Allen-Cahn equation is 
\begin{equation}
		\label{eqn:appendix_allen_eq}
		\begin{split}
			\partial_t u\left(\bmx,t\right) 
			=&\cL u +\cN u \\
			=&\epsilon^2 \nabla^2 u\left(\bmx,t\right) + \left(u\left(\bmx,t\right) - [u\left(\bmx,t\right)]^3\right), \quad  \bmx\in \Omega,\,t>0.
		\end{split}
	\end{equation}
As we have shown in main text section \ref{sec:ac},  the solution of linear and nonlinear parts in Eq.\eqref{eqn:appendix_allen_eq} can be written respectively as,
\begin{equation} \label{eq:appendix_linearac}
    u(\bmx,t+h) = \cW_h u(\bmx,t) = e^{h\cL}u(\bmx,t),
\end{equation}
\begin{equation}\label{eq:appendix_nonlinearac}
    u(\bmx,t+h) = \cS_h u(\bmx,t) = e^{h\cN}u(\bmx,t),
\end{equation}
where the linear operator $\cW_h =  e^{h\cL} $, and the nonlinear operator $\cS_h=e^{h\cN}$, with $h$ being the time step. 

We use the second order expansion of $\cW_h, \cS_h$, and $\cS_h^{-1}$ at $t=0$, for a small h:
\begin{equation} \label{eq:w_h}
    \cW_h u_0 = u_0 + h\partial_{xx} u_0 + \frac{h^2}{2} \partial_{xxxx} u_0 + \mathcal{O}(h^3),
\end{equation}
\begin{equation}\label{eq:s_h}
    \cS_h u_0 = u_0 + h (-u_0^3+u_0)+ \frac{h^2}{2}u_0 (1-4u_0^2+3u_0^4) + \mathcal{O}(h^3),
\end{equation}
\begin{equation}\label{eq:invert_s_h}
    \cS_{h}^{-1} u_0 = u_0 - h (-u_0^3+u_0)+ \frac{h^2}{2}u_0 (1-4u_0^2+3u_0^4) + \mathcal{O}(h^3).
\end{equation}
where $u(x,0)$ is denoted as $u_0$ for simplicity. 

We understand OSB by comparison with the traditional operator splitting algorithm. Suppose that one OSB is used to approximate the evolution from $u_0$ to $u_{T_{1}}$. Recall that in the operator splitting algorithm (Eq.\eqref{eq:gcOSN} in main text), $u_{T_{1}}$ is estimated by a stack of compositions of $N$ linear and nonlinear operators on $u_0$, and when $N$ is large enough, the approximation error is small. While in  OSB, we have $u_{T_1}\approx\psi_{\btheta}u_0 = \psi_{\cN_\alpha}\psi_{\cL_{\btheta}}u_{0}$, where $\btheta$ and $\alpha$ are all the learnable parameters. Thus, we have

\begin{equation} \label{app:error_u_t1}
    u_{T_1} = \cS_{h}\cW_{h}\dots\cW_{h}\cS_{h}\cW_{h}  u_{0} + \cO(Nh^2) \approx  \psi_{\cN_\alpha}\psi_{\cL_{\btheta}}u_{0}.  
\end{equation}
  
To better analyze the approximation error of the learnable layer $\psi_{\cL_{\btheta}}$ in the OSB, we set $\psi_{\cN_\alpha} = \overbrace{\cS_{h}\cdots\cS_{h}}^{N}$. Then for Eq.\eqref{app:error_u_t1} to hold, $ \psi_{\cL_{\btheta}}u_0$ can be written as

\begin{equation} \label{eq:appedix_linear_approx}
\begin{split}
\centering
        \psi_{\cL_{\btheta}}u_0 = &\psi_{\cN_\alpha}^{-1}\overbrace{\cS_{h}\cW_{h}\dots\cW_{h}\cS_{h}\cW_{h}}^{2N} u_{0}\\
        =&\overbrace{\cS_{h}^{-1}\cdots\cS_{h}^{-1}}^{N}\overbrace{\cS_{h}\cW_{h}\dots\cW_{h}\cS_{h}\cW_{h}}^{2N} u_{0} \\
        =& \cW_{Nh}u_0 - \cS_{h}^{-1}[\cS_h,\cW_{Nh}]\cW_h u_0 -  \cS_{h}^{-1}\cS_{h}^{-1}[\cS_h,\cW_{(N-2)h}]\cW_h \cS_{h}\cW_hu_0 \\
        & -  \cdots - \underbrace{\cS_{h}^{-1}\cdots\cS_{h}^{-1}}_{N-1}[\cS_h,\cW_{h}]\underbrace{\cW_h \cdots\cS_{h}\cW_h}_{2N-1} u_0 .
\end{split}
\end{equation}

In Allen-Cahn equation, $\cW_h$ and $\cS_h$ are not commutative, and their Lie bracket can be calculated as
\begin{equation}
   [\cS_h,\cW_h] u_0 = \cS_h\cW_h u_0 - \cW_h\cS_h u_0 =   -6  u_0(\partial_x u_0)^2h^2 + \mathcal{O}(h^3)
\end{equation}
by Eq.\eqref{eq:w_h} and Eq.\eqref{eq:s_h}. Furthermore, 
\begin{equation}
    \left[[\cS_h,\cW_h],\cS_h^{-1}\right] =  \cO(h^3). 
\end{equation}
Therefore, in Allen-Cahn equation,  Eq.\eqref{eq:appedix_linear_approx} gives:

\begin{equation}\label{eq:appendix_linear_approx_final}
\begin{split}
\centering
        \psi_{\cL_{\btheta}}u_0 
         = \cW_{Nh}u_0 - [\cS_h,\cW_h]\left(\cW_{(N-1)h} + 2\cW_{(N-2)h} +\cdots (N-1)\cW_h\right)u_0 +\cO(h^3)
\end{split}
\end{equation}
Note that on the right-hand side (RHS) of Eq.\eqref{eq:appendix_linear_approx_final}, the first term is linear, while the second term is nonlinear since the term $[\cS_h,\cW_h] =   -6  u_0(\partial_x u_0)^2h^2 + cO(h^3)$ is nonlinear. In practice, $\psi_{\cL_{\btheta}}$ in DOSnet is expected to be a fully linear layer. Therefore, besides approximating the first linear term $\cW_{Nh}u_0$ on the RHS of Eq.\eqref{eq:appendix_linear_approx_final}, the linear layer in the DOSnet provides a linear approximation for the second nonlinear term on the RHS of Eq.\eqref{eq:appendix_linear_approx_final}.


\section{Stability of Long Time Evolution  in Schr\"{o}dinger Equation} \label{appendix:schodinger}
 In this example, we validate the structure of Autoflow from another point of view, the stability of linear Autoflow in a long time evolution. We implement our experiment on the Schr\"{o}dinger equation as,
 \begin{align}
    i u_t = -u_{xx}+V\left(x,t\right)u,\quad x \in [-\pi,\pi],\, t>0,
\end{align}
with potential $V \equiv 0$, which describes a pure phase rotation from 0s to 0.03s. The initial and boundary conditions are the same as those in Eq.~\eqref{advectioneq}. 

The experimental result in Fig.\ref{fig:schroinger_toy} shows that the Autoflow with unitary matrices can  ``stabilize"  the inference errors in long time evolution. It is noticed that the design of intermediate dimension in Autoflow promises the use of unitary matrices  which respect the property of Schr\"{o}dinger equation. In contrast, it is difficult to utilize unitary matrices in a standard Convolutional Neural Network (CNN) with higher dimensional feature space. This shows that our model is flexible and scalable to respect the property of original PDEs.

Specifically, in order to guarantee the long-range stability \cite{jing2017tunable,kiani2022projunn}, we employ unitary matrices \cite{jing2017tunable} in each layer of linear Autoflow to enforce the probability conservation property of the Schrodinger equation. 
We test the performance of long-time evolution on the two model settings: linear Autoflow with unitary matrices and regular linear Autoflow without unitary matrices. First, the model is trained to approximate the solution at time $T=0.03$s given input solution at 0s.  Next, we apply the trained model multiple times and inference the solutions on test data from T to 4T. The test results under the two model settings are shown in Fig.\ref{fig:schroinger_toy}. The results are examined by the relative error, $\frac{1}{N}\sum_{i=0}^{N}\left(\frac{\hat{u}(x_i,t)-u(x_i,t)}{u(x_i,t)}\right)^2$, where $\hat{u}$ is the prediction of the network and $u$ is the target. The results show that the Autoflow without restriction shows ``exponential" errors after being applied multiple times, while the Autoflow with unitary matrices can  ``stabilize"  the inference errors in a long time evolution. 
\begin{figure} 
    \centering
    \includegraphics[width=0.6\textwidth]{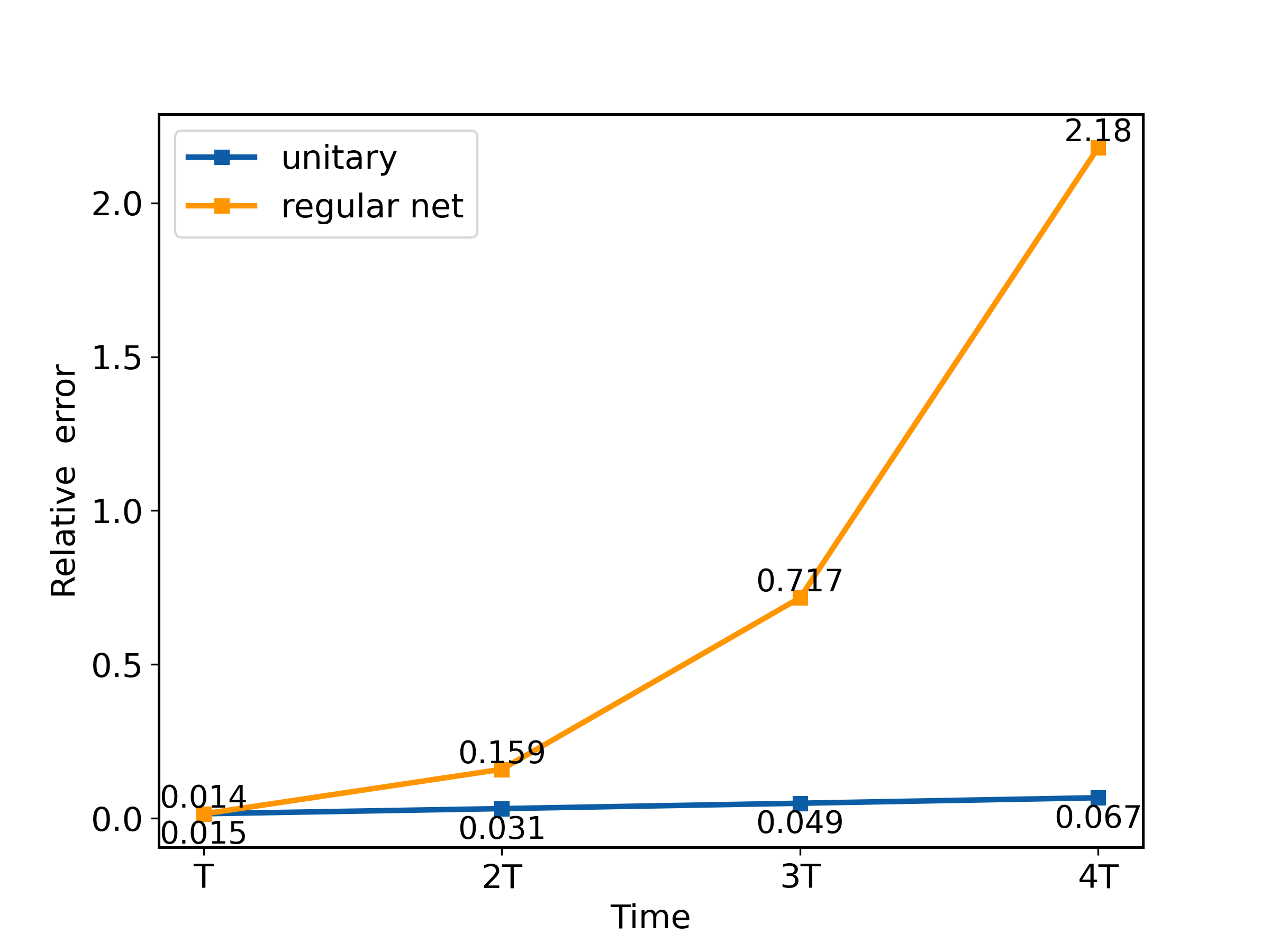}
    \caption{The relative errors of linear Autoflow with unitary matrices (blue line) and regular linear Autoflow (orange line) for long-time inference.}
    \label{fig:schroinger_toy}
\end{figure}

\section{Analysis of Behaviors of the Weights in DOSnet} \label{appendix:proof}
We analyze the converged weights in the linear DOSnet when it is applied to solve a linear PDE. 

In supervised learning, the dataset contains $M$ input-target pair $\{u_0^{(i)},u_T^{(i)}\},\,i=1,\cdots,M$.
Consider a two-layer linear neural network\cite{saxe2013exact} to learn the mappling: $f: u_0\in \mathbb{R}^{N_1} \rightarrow u_T \in \mathbb{R}^{N_3}$. The output  $o = W^{32}h =W^{32}W^{21}u_0 \in \mathbb{R}^{N_3}$ is trained to approximate $u_T$, where the hidden state $h = W^{21}u_0 \in \mathbb{R}^{N_2}$. In the linear Autoflow, $N_1 = N_2 = N_3$. The loss function is defined as 
\begin{equation} \label{eq:optimize_loss}
    L = \sum_{i=1}^{M}\lVert u_T^{(i)} - W^{32}W^{21}u_0^{(i)}\rVert^2+\beta \left(\lVert W^{32}\rVert^2 + \lVert W^{21}\rVert^2\right),
\end{equation}

we explore the behavior of the weights $W^{32}$ and $W^{21}$ by analyzing the gradient descent dynamics that optimizes Eq.\eqref{eq:optimize_loss}:
\begin{equation} \label{eq:limitw21}
    \frac{1}{\lambda}\frac{dW^{21}}{dt} = (W^{32})^{T}\left(\Sigma^{31}-W^{32}W^{21}\Sigma^{11}\right)-\beta W^{21},
\end{equation}
\begin{equation} \label{eq:limitw32}
    \frac{1}{\lambda}\frac{dW^{32}}{dt} = \left(\Sigma^{31}-W^{32}W^{21}\Sigma^{11}\right)(W^{21})^{T}-\beta W^{32},
\end{equation}
where  $t$ measures time in the unit of iterations, $\lambda$ is the step size, $\Sigma^{11} = \frac{1}{M}\sum_{i = 1}^{M}u_0^{(i)}u_0^{(i)T}$  is the $N_1\times N_1$ input covariance matrix, $\Sigma^{31} = \frac{1}{M}\sum_{i = 1}^{M}u_T^{(i)}u_0^{(i)T}$ is the $N_3\times N_1$ input-output cross-covariance matrix. 

We obtain the following proposition for the behaviors of the weights in the linear Autoflow. This proposition holds when the linear Autoflow is applied to a general problem including solving a linear PDE. Note that $N_1=N_2=N_3$ in these propositions.

\begin{proposition}\label{thm:weight_relation} 
Assume the initial weights follow orthogonal initialization from \cite{saxe2013exact}, then the converged weight matrices $W^{32}$ and $W^{21}$ has the following  relation, 
\begin{equation} \label{eq:relation_w21_w32}
    W^{21} = \text{sgn}(S^{31}) (W^{32})^{T} U^{33}(V^{11})^{T},
\end{equation}
\begin{equation} \label{eq:fixed_point_w21}
    W^{21} = \pm R\sqrt{\text{sgn}(S^{31})S^{31}-\beta} (V^{11})^{T},
\end{equation}
\begin{equation}\label{eq:fixed_point_w32}
    W^{32} = \pm\text{sgn}(S^{31}) U^{33}\sqrt{\text{sgn}(S^{31})S^{31}-\beta} R^T, 
\end{equation}
where $S^{31}$ is the diagonal matrix, $V^{11}$ and $U^{33}$ are orthogonal matrices from the singular value decomposition (SVD) of $\Sigma^{31}$, and $R$ is an arbitrary orthogonal matrix.
\end{proposition}

\begin{proof}
Consider the dynamics of weights given by Eqs.\eqref{eq:limitw21} and \eqref{eq:limitw32}. 
Following the assumption in \cite{saxe2013exact}, we assume $\Sigma^{11}=I$, and  consider the SVD decomposition of $\Sigma^{31}$:
\begin{equation} \label{eq:svd}
    \Sigma^{31} = U^{33}S^{31}(V^{11})^{T} = \sum_{\alpha=1}^{N_1} s_\alpha u_\alpha v_\alpha^T,
\end{equation}
where $S^{31} $ is a $N_3\times N_1$ diagonal matrix whose diagonal elements are singular values $s_\alpha$, $\alpha = 1,\cdots,N_1$, $s_1\geq s_2 \geq \cdots \geq s_{N_1}$, and $V^{11}$ and $U^{33}$ are orthogonal matrices.

By changes of variable $W^{21} = \Bar{W}^{21}(V^{11})^{T}, W^{32} = U^{33}\Bar{W}^{32}$, Eqs.\eqref{eq:limitw21} and \eqref{eq:limitw32} can be rewritten as
\begin{equation} \label{eq:limit_change_w21}
    \Gamma\frac{d\Bar{W}^{21}}{dt} = (\Bar{W}^{32})^{T}\left(S^{31}-\Bar{W}^{32}\Bar{W}^{21}\right)-\beta \Bar{W}^{21},
\end{equation}
\begin{equation} \label{eq:limit_change_w32}
    \Gamma\frac{d\Bar{W}^{32}}{dt} = \left(S^{31}-\Bar{W}^{32}\Bar{W}^{21}\right)\Bar{W}^{21T}-\beta \Bar{W}^{32},
\end{equation}
where $\Gamma = \frac{1}{\lambda}$.

Denote $\bm{a}^{\alpha}$ as the $\alpha\text{-}th$ column of $\Bar{W}^{21}$ and $\bm{b}^{\alpha T}$ the $\alpha\text{-}th$ row of $\Bar{W}^{32}$.  We rewrite Eqs.\eqref{eq:limit_change_w21} and \eqref{eq:limit_change_w32} by the column vectors of the two weight matrices,
\begin{equation} \label{eq:limit_column_w21}
    \Gamma\frac{d \bm{a}^{\alpha}}{dt} = \left(s_\alpha-\bm{a}^{\alpha}\cdot\bm{b}^{\alpha}\right)\bm{b}^{\alpha} - \sum_{\gamma \neq \alpha}\bm{b}^{\gamma}\left(\bm{a}^{\alpha}\cdot\bm{b}^{\gamma}\right) - \beta \bm{a}^{\alpha},
\end{equation}
\begin{equation} \label{eq:limit_column_w32}
    \Gamma\frac{d \bm{b}^{\alpha}}{dt} = \left(s_\alpha-\bm{a}^{\alpha}\cdot\bm{b}^{\alpha}\right)\bm{a}^{\alpha} - \sum_{\gamma \neq \alpha}\bm{a}^{\gamma}\left(\bm{b}^{\alpha}\cdot\bm{a}^{\gamma}\right) - \beta \bm{b}^{\alpha}.
\end{equation}
The energy function can be written accordingly as,
\begin{equation}\label{eq:energy}
    E = \frac{1}{2\Gamma}\sum_\alpha \left(s_\alpha -\bm{a}^{\alpha}\cdot\bm{b}^{\alpha}\right)^2+\frac{1}{2\Gamma}\sum_{\alpha \neq \gamma} \left(\bm{a}^{\alpha}\cdot\bm{b}^{\gamma}\right)^2 + \frac{\beta}{2\Gamma} \sum_\alpha \left(\left(\bm{a}^{\alpha}\right)^2+\left(\bm{b}^{\alpha}\right)^2\right).
\end{equation}

In order to decouple the interaction terms in Eqs.\eqref{eq:limit_column_w32} and \eqref{eq:limit_column_w21}, following the assumption on initial conditions in \cite{saxe2013exact}, we assume that the initial value of $\bm{a}^\alpha$ and $\bm{b}^\alpha$ are all parallel to $\bm{\gamma}^\alpha$, i.e., $\bm{a}^\alpha \mathop{//} \bm{b}^\alpha \mathop{//}\bm{\gamma}^\alpha $, where $\left\{ \bm{\gamma}^\alpha \right\}$ is a fixed collection of $N_2$-vectors that form an orthonormal basis. Under this assumption, $\bm{a}^\alpha$ and $\bm{b}^\alpha$ are in the same direction and only differ in their scalar magnitudes, and are orthogonal to each other.
Consider the magnitude of $\bm{a}^\alpha$ and $\bm{b}^\alpha$ on $\bm{\gamma}^\alpha$, $a^{\left(\alpha\right)} = \bm{a}^\alpha\cdot\bm{\gamma}^\alpha$, $b^{\left(\alpha\right)} = \bm{b}^\alpha\cdot\bm{\gamma}^\alpha$. The dynamics of the scalar magnitudes are,

\begin{equation} \label{eq:simlfy_adynamics}
    \Gamma \frac{d a^{\left(\alpha\right)}}{dt} = b^{\left(\alpha\right)}\left(s_\alpha-a^{\left(\alpha\right)}b^{\left(\alpha\right)}\right)-\beta a^{\left(\alpha\right)},
\end{equation}

\begin{equation}\label{eq:simlfy_bdynamics}
        \Gamma \frac{d b^{\left(\alpha\right)}}{dt} = a^{\left(\alpha\right)}\left(s_\alpha-a^{\left(\alpha\right)}b^{\left(\alpha\right)}\right)-\beta b^{\left(\alpha\right)}.
\end{equation}
The energy function is,
\begin{equation} \label{eq:energy-orthogonal}
    E\left(a^{\left(\alpha\right)},b^{\left(\alpha\right)}\right) = \frac{1}{2\Gamma}\left(s_\alpha-a^{\left(\alpha\right)}b^{\left(\alpha\right)}\right)^2+\frac{\beta}{2\Gamma}\left(a^{\left(\alpha\right)2}+b^{\left(\alpha\right)2}\right)
\end{equation}

Thus the fixed points should satisfy 
\begin{equation}
\left\{
\begin{aligned}
& a^{\left(\alpha\right)}b^{\left(\alpha\right)}\pm\beta = s_\alpha, \\
& a^{\left(\alpha\right)} = \pm b^{\left(\alpha\right)}.
\end{aligned}
\right.
\end{equation}
Thus the fixed points have the following three cases: 
\begin{itemize}
    \item Case 1: $(a^{(\alpha)},b^{(\alpha)})= \left(0,0\right)$
    \item  Case 2: when $a^{\left(\alpha\right)}=-b^{\left(\alpha\right)}, a^{\left(\alpha\right)}b^{\left(\alpha\right)}-\beta = s_\alpha$, two fixed points are  $\left(-\sqrt{-s_\alpha-\beta},\sqrt{-s_\alpha-\beta}\right)$, $\left(\sqrt{-s_\alpha-\beta},-\sqrt{-s_\alpha-\beta}\right)$.
    \item  Case 3: when $a^{\left(\alpha\right)}=b^{\left(\alpha\right)}, a^{\left(\alpha\right)}b^{\left(\alpha\right)}+\beta = s_\alpha$, two fixed points are $\left(-\sqrt{s_\alpha-\beta},-\sqrt{s_\alpha-\beta}\right)$,
    $\left(\sqrt{s_\alpha-\beta},\sqrt{s_\alpha-\beta}\right)$.
\end{itemize}

To examine the convergence of the dynamics of Eqs.\eqref{eq:simlfy_adynamics} and \eqref{eq:simlfy_bdynamics}, we examine the stability of the above fixed points by linear stability analysis.
Here we set $\beta > 0, \Gamma = 1$. In case 1, the linearization of Eqs.\eqref{eq:simlfy_adynamics} and \eqref{eq:simlfy_bdynamics} is
\begin{equation}
    \frac{d}{dt}
    \begin{pmatrix}\Delta a^{\left(\alpha\right)}\\\Delta b^{\left(\alpha\right)}\end{pmatrix} = \begin{pmatrix}-\beta & s_\alpha \\s_\alpha & -\beta \end{pmatrix} \begin{pmatrix} \Delta a^{\left(\alpha\right)}\\\Delta b^{\left(\alpha\right)}
    \end{pmatrix},
\end{equation} 
where $\Delta a$ and $\Delta b$ are deviations from the fixed point a and b, respectively.
Eigenvalues of the linearized operator $\big(\begin{smallmatrix}
  -\beta & s_\alpha\\
  s_\alpha & -\beta
\end{smallmatrix}\big)$ are $-\beta-s_\alpha$ and $s_\alpha-\beta$. Thus, when $\lvert s_{\alpha}\rvert < \beta$, the fixed point $(0,0)$ is linearly stable.
For case 2, the linearized operator has the form
$\big(\begin{smallmatrix} s_\alpha & s_\alpha+2\left(-s_\alpha-\beta\right)\\s_\alpha+2\left(-s_\alpha-\beta\right) & s_\alpha \end{smallmatrix}\big)$.
The eigenvalues are $-2\beta$ and $2\left(s_\alpha+\beta\right)$. When $\left(s_\alpha+\beta \right) < 0$, i.e. $s_\alpha<-\beta$, the fixed point $a = -b$ is linearly stable.
For case 3, the eigenvalues of the linearized operator $\big(\begin{smallmatrix} -s_\alpha & s_\alpha-2\left(s_\alpha-\beta\right)\\s_\alpha-2\left(s_\alpha-\beta\right) & -s_\alpha  \end{smallmatrix}\big)$ are $-2\beta$ and $-2\left(s_\alpha-\beta\right)$. Thus, when $s_\alpha > \beta$, the fixed point $a = b$ is linearly stable. We show the dynamics of Eqs.\eqref{eq:simlfy_adynamics} and \eqref{eq:simlfy_bdynamics} in the above three cases in Fig.\ref{fig:vector_fields_3_figs}, .

\begin{figure}
     \centering
      \begin{subfigure}[b]{0.325\textwidth}
         \centering
         \includegraphics[width=\textwidth]{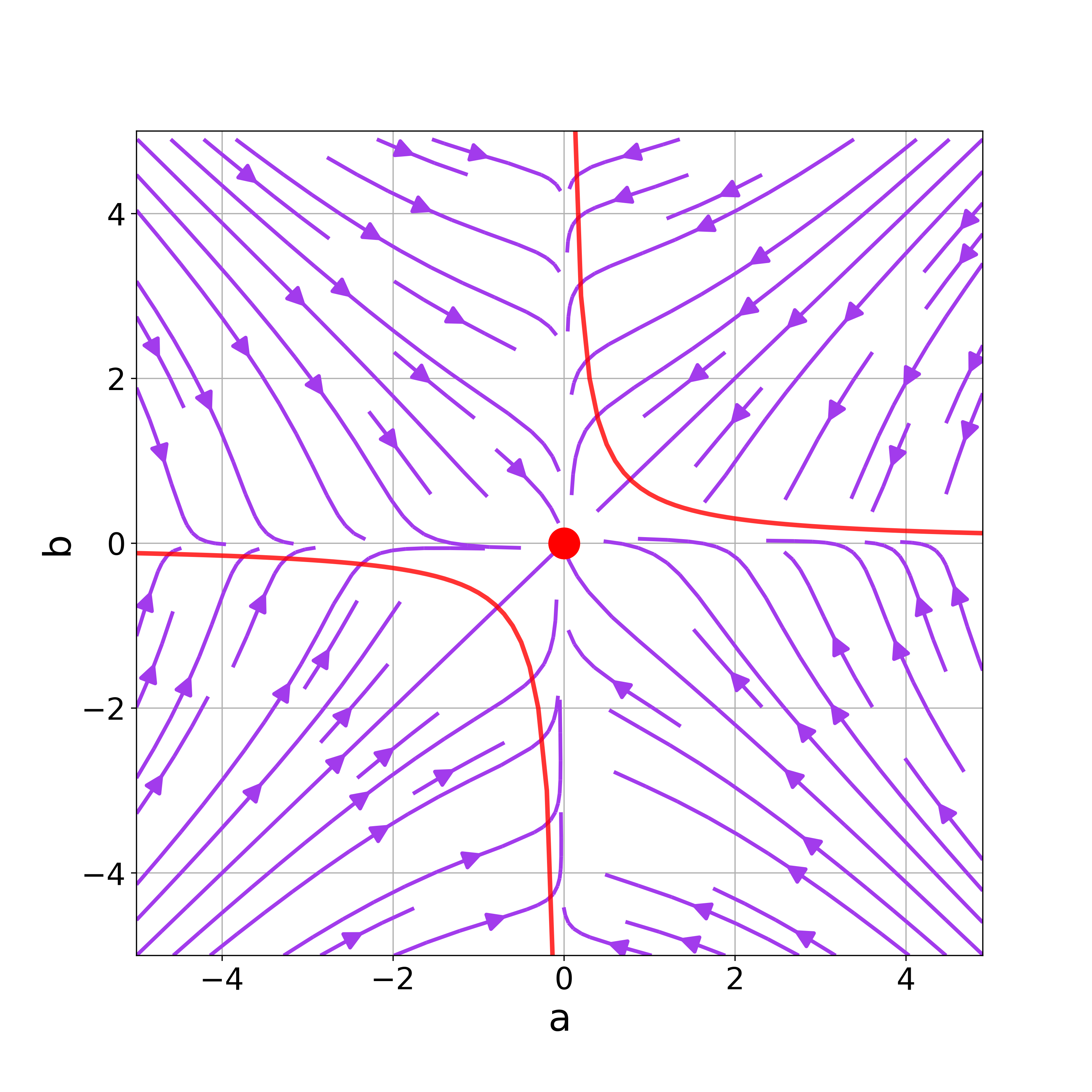}
         \caption{}
         \label{fig:vector_field_origin_point}
     \end{subfigure}
     \hfill
     \begin{subfigure}[b]{0.325\textwidth}
         \centering
         \includegraphics[width=\textwidth]{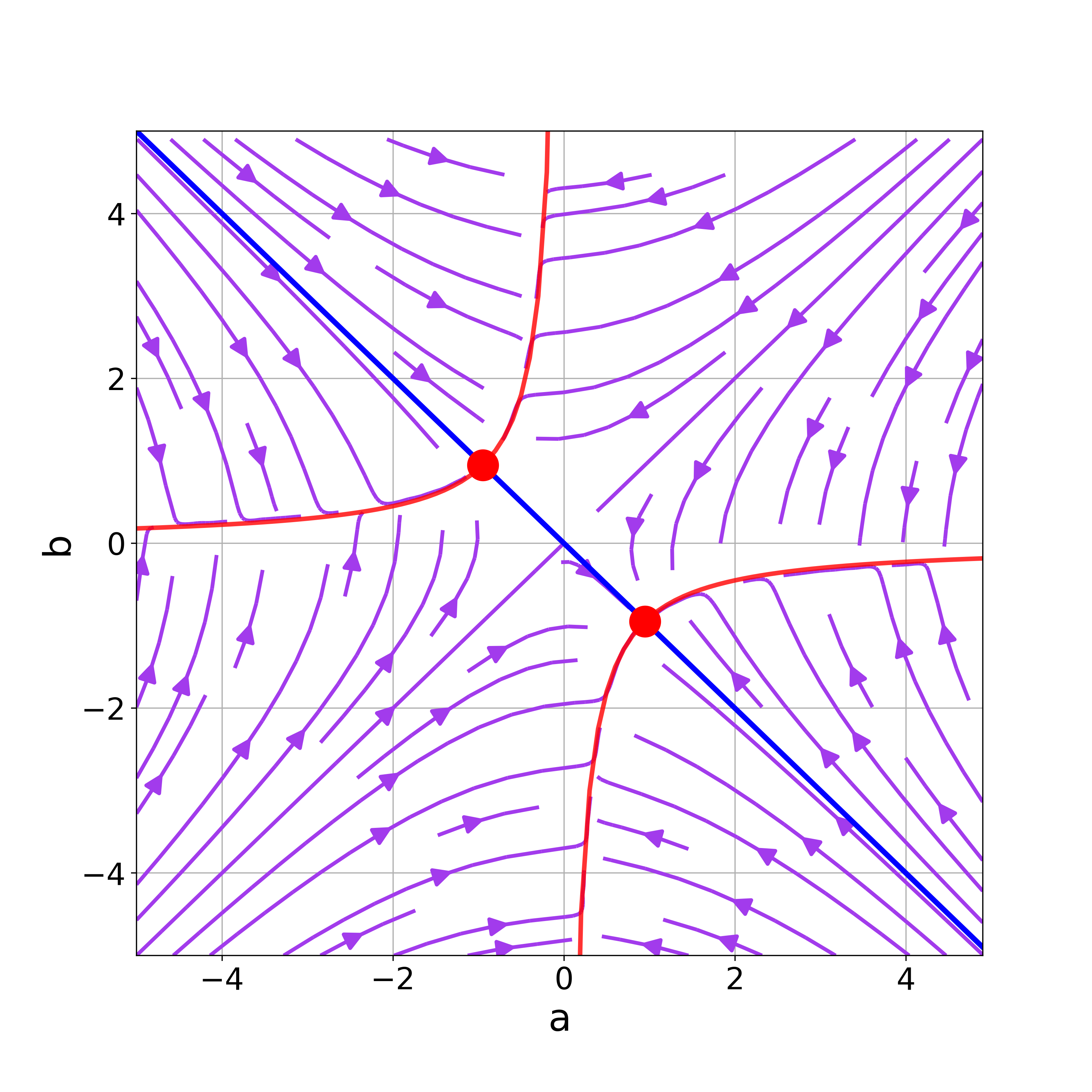}
         \caption{}
         \label{fig:vector_field_s=-1}
     \end{subfigure}
     \hfill
     \begin{subfigure}[b]{0.325\textwidth}
         \centering
         \includegraphics[width=\textwidth]{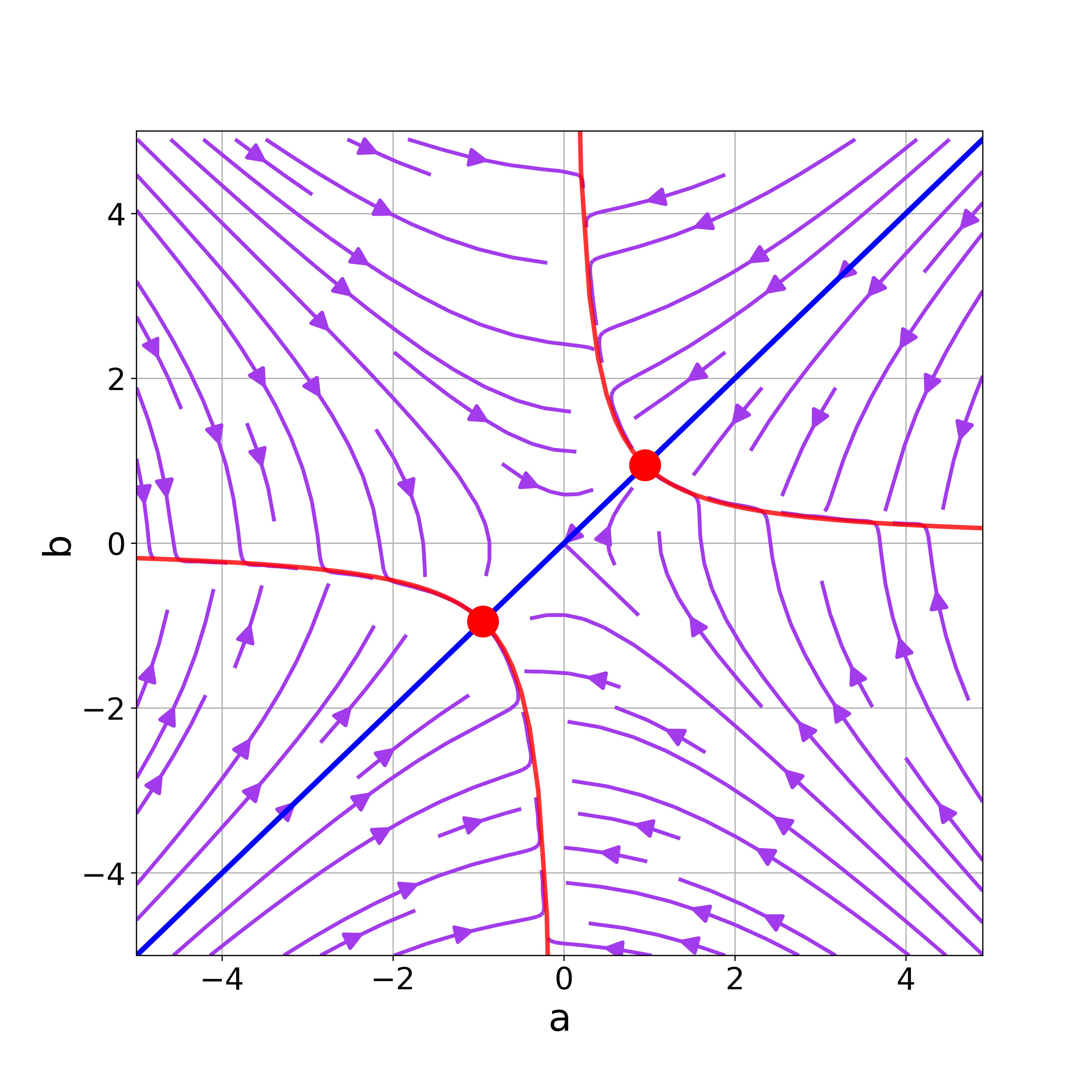}
         \caption{}
         \label{fig:vector_field_s=1}
     \end{subfigure}
        \caption{Vector field defined by Eqs.\eqref{eq:simlfy_adynamics} and \eqref{eq:simlfy_bdynamics}. (a) Case 1: vector field with $\beta=0.5, s=0.1, \Gamma = 1.$ (b) Case 2: vector field with $\beta=0.1, s=-1, \Gamma = 1.$ (c) Case 3: vector field with $\beta=0.1, s=1, \Gamma = 1.$ The red curves indicate $s = ab \pm \beta$, and the blue straight lines show $a = \pm b$. The red points are the stable fixed points.}
        \label{fig:vector_fields_3_figs}
\end{figure}

In summary, for any $\alpha$, when $\lvert s_{\alpha}\rvert > \beta$, the nontrivial fixed points will be linearly stable. Especially, $\bm{a}^{\alpha}$ = $\text{sgn}(s_\alpha) \bm{b}^{\alpha}$.
In contrast, when $\lvert s_{\alpha}\rvert < \beta$, $\bm{a}^\alpha$ and $\bm{b}^\alpha$ will converge to zero. 

In the nontrivial case, we have $\Bar{W}^{21} = \text{sgn}(S^{31})(\Bar{W}^{32})^{T}$.  Furthermore, by the assumption that the columns of $\Bar{W}^{21}$ and the rows of $W^{32}$ satisfy $\bm{a}^\alpha \mathop{//} \bm{b}^\alpha \mathop{//}\bm{\gamma}^\alpha, \alpha = 1,\cdots,N_2$, we can write $\Bar{W}^{21}$ and $\Bar{W}^{32}$ according to their converged magnitudes $(a^{(\alpha)},b^{(\alpha)})$ in Case 2 and 3 as
\begin{equation}
    \Bar{W}^{21} = \pm R\sqrt{\text{sgn}(S^{31})S^{31}-\beta},
\end{equation}
\begin{equation}
    \Bar{W}^{32} = \pm\text{sgn}(S^{31}) \sqrt{\text{sgn}(S^{31})S^{31}-\beta} R^T,
\end{equation}
where $R$ is the orthogonal matrix whose $\alpha$-th column is $\bm{\gamma}_\alpha, \alpha = 1,\cdots,N_2.$
Thus  $W^{21}$ and $W^{32}$ satisfies Eqs.\eqref{eq:relation_w21_w32}-\eqref{eq:fixed_point_w32}.
\end{proof}

Consider a linear PDE with some symmetry properties. Suppose the linear PDE $\cL u(x,t)=0$, where $\cL$ is a differential operator, is defined on a periodic spatial domain with period $L$, and $t\geq 0$. The initial condition is $u(x,0)=g(x)$.
The fundamental solution $\phi(x,x';t,t')$ of this PDE  satisfies
\begin{equation}\label{app:fundamental_sol}
\begin{split} 
        \cL \phi(x,x';t,t') &= 0, \\
    \text{s.t.} \quad \phi(x,x';t',t')&=\delta (x-x'), 
\end{split}
\end{equation}
where $\delta$ is the Dirac Delta function. Consider the following symmetry properties of the PDE.
\begin{itemize}
    \item \textbf{Spatial translational symmetry}. Then  we can derive $\cL u(x+\delta x,t)=0$ from $\cL u(x,t)=0$, where $\delta x$ is an arbitrary distance, and the fundamental solution $\phi$ given in Eq.\eqref{app:fundamental_sol} satisfies $\phi(x,x';t,t') = \phi(x-x';t,t')$.
    \item \textbf{Temporal translational symmetry}. We can derive $\cL u(x,t+\delta t)=0$ from $\cL u(x,t)=0$, where $\delta t$ is an arbitrary time, and the fundamental solution $\phi$ given in Eq.\eqref{app:fundamental_sol} satisfies $\phi(x,x';t,t') = \phi(x,x';t-t')$.
\end{itemize} 

With the spatial and temporal translational symmetry, the fundamental solution given in Eq.\eqref{app:fundamental_sol} satisfies $\phi(x,x';t',t') = \phi(x-x',t-t')$.
\begin{itemize}
    \item \textbf{Even parity symmetry}. We can derive $\cL u(-x,t)=0$ from $\cL u(x,t)=0$, and the fundamental solution $\phi$ satisfies $\phi(x-x',t-t') = \phi(x'-x,t-t')$.
    \item \textbf{Time-reversal symmetry}. We can derive $\cL u(x,-t)=0$ from $\cL u(x,t)=0$, and the fundamental solution $\phi$ satisfies $\phi(x-x',t-t') = \phi(x-x',t'-t)$.
\end{itemize}

The following proposition summarizes the behaviors of the weights in the linear Autoflow when it is applied to solve a linear PDE with some symmetry properties. Especially, under some symmetries of the PDE, the weights between different layers in the network are equal. 

\begin{proposition}[] \label{thm:pde_condition}
    Consider the initial value problem of a linear PDE with periodic boundary condition in spatial domain as described above. The spatial coordinate $x$ is discretized into  $N_1$ uniform grid points, and the initial condition $g(x)$ of the PDE on the discrete mesh satisfies the covariance matrix $\Sigma^{11}=I$. If this linear PDE satisfies the spatial and temporal translational symmetries, and the even parity symmetry ( or the time-reversal symmetry), then $\Sigma^{31}$, the input-output cross-covariance matrix between $u(x_j,0)$ and $u(x_k,T)$ where $j,k = 1,2,\cdots,N_1$, is symmetric up to the leading order of $\Delta x$, where $\Delta x$ is the grid constant. As a result, we can specify $U^{33} = V^{11}$, thus $W^{21}$ and $ W^{32}$ in the linear Autoflow for solving this PDE satisfy $W^{21} = \text{sgn}(S^{31}) (W^{32})^{T}$. Moreover, when $W^{32}$ is symmetric, we have $W^{21} = \text{sgn}(S^{31}) W^{32}$.
\end{proposition}

\begin{proof}
    By Eq.\eqref{eq:svd}, $U^{33}$ and $V^{11}$ are determined by the input-output cross-covariance matrix $\Sigma^{31}$. In the case being considered, the input and output data are the solutions of the linear PDE at different time levels. Assume that the input data is the solution of the linear PDE at initial time $t=0$: $u\left(x,0\right)=g\left(x\right)$, where $x$ is spatial coordinate. Ground truth at time T is the solution $u\left(x, T\right)$. The cross-covariance between $u(x_j,t_1)$ and $u(x_k,t_2)$, where $j,k = 1,2,\cdots,N_1$ and $t_1, t_2 \in [0,T] $ is $\left< u(x_j,t_1), u(x_k,t_2)\right > = \frac{1}{M}\sum_{i = 1}^{M}u(x_j,t_1)^{(i)}u(x_k,t_2)^{(i)}$, where $M$ is the number of data.  Specifically, the cross-covariance matrix between $u(x_j,0)$ and $u(x_k,T)$ is
\begin{equation} \label{eq:correlation}
\begin{aligned}
       \Sigma^{31}_{jk} = & \left<  u(x_j,0),u(x_k,T)\right>\\
 \approx & \left< g(x_j),\sum_{\ell=1}^{N_1}\phi(x_\ell,x_k;T,0)g(x_\ell)\Delta x\right >  \\
 =&\sum_{\ell=1}^{N_1}\phi(x_\ell,x_k;T,0)\left< g(x_j),g(x_\ell)\right > \Delta x \\
 =& \sum_{\ell=1}^{N_1} \phi(x_\ell,x_k;T,0)\delta_{jl}\Delta x   \\
 =&\phi(x_j,x_k;T,0) \Delta x ,  \\
\end{aligned}
\end{equation}
where $\phi$ is the fundamental solution of the linear PDE. The assumption of $\Sigma^{11} = I$, means that the initial data $g(x)$ satisfies $\left< g(x_j),g(x_\ell)\right > = 
\delta_{jl} = \left\{
\begin{array}{lr} 0, \, \text{if} \, j \neq l ,\\ 1, \,  \text{if} \, j=l. \end{array}
\right.$ Since the linear PDE satisfies the spatial and temporal translational symmetries, its fundamental solution has the property that $\phi(x_j,x_k; T,0) = \phi (x_j-x_k, T)$. 

There are two cases in which we derive the symmetry of $\Sigma^{31}$: (i) if the linear PDE satisfies the even parity symmetry, then we have $\phi (x_j-x_k, T) = \phi (x_k-x_j, T)$; (ii) if the linear PDE has time-reversal symmetry, according to the definition of cross-covariance, we have $\left<  u(x_j,0),u(x_k, T)\right> = \phi (x_j-x_k, T) = \left< u(x_k, T), u(x_j,0)\right> = \phi (x_k-x_j,-T) = \phi (x_k-x_j, T)$.
By the above conditions, we can derive that the input-output cross-covariance matrix $\Sigma^{31}_{jk}$ has the symmetric leading order behavior $\phi(x_j,x_k;T,0)\Delta x,$ where $j,k = 1,\cdots,N_1$, as $\Delta x \rightarrow 0$. Thus we can specify $U^{33} = V^{11}$. Further using Eq.\eqref{eq:relation_w21_w32}, we have $W^{21} = \text{sgn}(S^{31}) (W^{32})^{T}$. Especially, when $W^{32}$ is symmetric, i.e., $W^{32} = (W^{32})^T$, then $W^{21} = \text{sgn}(S^{31})W^{32}$.
\end{proof}

The following proposition gives a necessary condition for the intermediate output in linear Autoflow to be a true intermediate state of the solution of the linear PDE. 
\begin{proposition}\label{thm: weight_symmtric}
If the intermediate output $h$ of the two-layer linear Autoflow is the solution of the PDE in Proposition \ref{thm:pde_condition}, then $W^{32}$ and $W^{21}$ should be symmetric.
\end{proposition}
\begin{proof}
If the intermediate output $h$ of the linear Autoflow is the solution of the linear PDE in Proposition \ref{thm:pde_condition}, then the conclusion in Proposition \ref{thm:pde_condition} holds for $h$. That means the input-intermediate output cross-covariance matrix $\Sigma^{21} = \frac{1}{M}\sum_{i = 1}^{M}h^{(i)}u_0^{(i)T}$ should be symmetric. By specifying $R=U^{33}=V^{11}$, $W^{32}$ and $W^{21}$ are symmetric.

\end{proof}

\section{A Numerical Scheme of Solving Allen-Cahn Equation}\label{app:numerical_ac}
Numerically, the operator splitting method is commonly used to solve the Allen-Cahn equation \eqref{eqn:allen_eq}. Denoted $u_n = u\left(\cdot, t_n\right)$, and $t_n = n\tau$ where $\tau$ is the time step. In each time step, the symmetrized Strang splitting scheme \cite{strang1968construction} for solving the Allen-Cahn equation consists the following three sub equations successively,
\begin{align}
 &\partial_t u_1 =\epsilon^2 \nabla^2 u_1 \tag{\ref{eqn:allen_eq}{a}} \,\,\, \text{in}\, \left(t_n,t_{n+1/2}\right), \, u_1\left(t_n\right) = u_n \rightarrow u_1^\star = u_1\left(t_{n+1/2}\right)  \label{eqn:allen_eqa} \\
 &\partial_t u_2 =-\left(u_2^3-u_2\right) \tag{\ref{eqn:allen_eq}{b}} \,\,\, \text{in}\, \left(t_n,t_{n+1}\right), \, u_2\left(t_n\right) = u_1^\star \rightarrow u_2^\star = u_2\left(t_{n+1}\right) \label{eqn:allen_eqb} \\
 &\partial_t u_3 =\epsilon^2 \nabla^2 u_3 \tag{\ref{eqn:allen_eq}{c}} \,\,\, \text{in}\, \left(t_{n+1/2},t_{n+1}\right),\, u_3\left(t_{n+1/2}\right) = u_2^\star \rightarrow u\left(t_{n+1}\right) = u_3\left(t_{n+1}\right).  \label{eqn:allen_eqc} 
\end{align}
The diffusion equations \eqref{eqn:allen_eqa} and \eqref{eqn:allen_eqc} can be solved by the fast Fourier transform (FFT) method with computational complexity $O\left(N^2\text{log}N\right)$, where $N^2$ is number of discretized grid in the two dimensional spatial domain. The second nonlinear equation \eqref{eqn:allen_eqb} can be solved analytically as,
\begin{equation} \label{eq: ac-nonlinear}
	 u_2\left(t_{n+1}\right) = u_1^\star \Big{/}\sqrt{e^{-2\tau}+\left(1-e^{-2\tau}\right)\left(u_1^\star\right)^2}. 
\end{equation}
The truncated error of this numerical scheme is $O\left(\tau^2\right)$. In order to reduce the error, the time step $\tau$ has to be very small, leading to a large number  of steps. The FFT method in equation \eqref{eqn:allen_eqa} and \eqref{eqn:allen_eqc} gives the largest contribution to the computational complexity during the numerical simulation.

\section{Optical transmission system} \label{app:nlse}
\begin{figure}[h]
	\centering
	\includegraphics[width=0.8\linewidth]{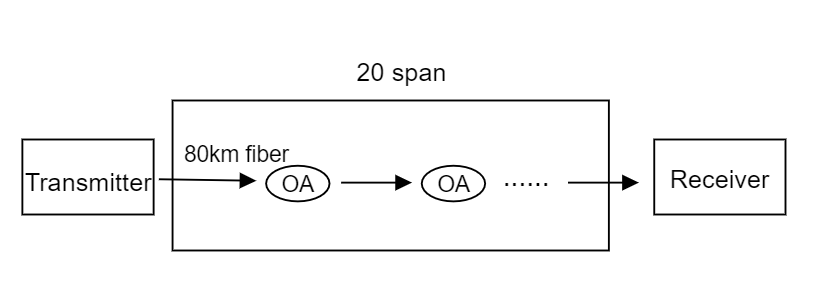}
	\caption{An optical transmission system (signal channel). There are three
basic components in an optical transmission system: transmitter, optical fiber,
and receiver. For the optical fiber with 20 spans, each span of fiber is 80 km and contains an inline optical amplifier (OA). }
	\label{fig:fibertransmission}
\end{figure}
Here we give a more detailed description of the optical transmission system \cite{agrawal_fiber-optic_2010} which is studied as an application of DOSnet in section \ref{sec:nlse}. In our experiments, we only consider the single channel system. 
There are three basic components in an optical transmission system: transmitter, optical fiber, and receiver. In the beginning, the electrical signal is converted into optical signal at the transmitter. Then the optical signal propagates through the fiber for a long distance. Finally it arrives at the optical receiver. The received signal is then converted back to electrical domain and is processed by digital signal processor (DSP). The processed signal is used for information detection. 

The clean signal at transmitter propagate through 20 spans optical fiber and arrived at the receiver. Each span of fiber is 80 km and contains an inline optical amplifier (OA). OA induces  amplified spontaneous emission noise. This noise, together with the three factors in the fiber, loss, dispersion, and nonlinear interaction, cause the signal to be distorted inside optical fiber. 
In order to recovery the original information of this signal, nonlinear compensation is required before the decision procedure is applied. The whole pipeline is shown in Fig.\ref{fig:fibertransmission}. 

Next, we describe each component of the system in details.

\textbf{Signal generated at transmitter.} 
The input signal at the transmitter takes the form of
\begin{align}\label{eq:input_signal}
    u_0(t) = \sqrt{P_L}\sum_{n} a_n h(t-n T_s),
\end{align}
where $P_L$ is the called launch power. We choose launch power as $0$ dBm here, i.e., $P_L = 10^{-3}$ W. Each $a_n\in\cC$ represents a symbol and $h$ is the root-raised-cosine filter. Here $\cC:=\{\pm(2k+1)\pm i (2l+1)\}_{0\le k,l\le1}\subset\bbC$ is a set of 16 grid points, called 16QAM constellation. The real function $h$ is called root-raised-cosine filter, which is defined by

\begin{equation}\label{pulse-shape}
 h\left(t\right)=\left\{
\begin{aligned} 
	&\frac{1}{T_s}\left(1+\rho\left(\frac{4}{\pi}-1\right)\right), \quad &t= 0 \\
	&\frac{\rho}{T_s\sqrt{2}}[\left(1+\frac{2}{\pi}\right)\sin\left(\frac{\pi}{4\rho}\right)+\left(1-\frac{2}{\pi}\right)\cos\left(\frac{\pi}{4\rho}\right)],\quad &t=\pm \frac{T_s}{4\rho} \\
	&\frac{1}{T_s}\frac{\sin[\pi\frac{t}{T_s}\left(1-\rho\right)]+4\rho\frac{t}{T_s}\cos[\pi\frac{t}{T_s}\left(1+\rho\right)]}{\pi\frac{t}{T_s}[1-\left(4\rho\frac{t}{T_s}\right)^2]}, \quad &\text{otherwise}
\end{aligned}
\right.
\end{equation}
where $\rho=0.1$ and $T_s$ is the reciprocal of the symbol-rate. 

\textbf{Signal propagation at fiber.} Optical fiber is a fundamental component in optical transmission systems. Our 1600 km transmission system consists of 20 spans of 80 km fibers. At the end of each span, optical amplifier (OA) is used to exactly recover the amplitude of the signal. Besides, OA also add to the signal a white Gaussian noise with zero mean and variation $\sigma^2 = Fh\nu_0\left(G-1\right)\Delta\nu$, where $h$ is the Planck's constant, $\nu_0$ denotes the carrier frequency, and $G$ is the gain of amplifier. Here we choose noise figure $F=4.5$ dB, $\Delta\nu = 50 \text{GHz}$ where samples per symbol (sps) is 4.

\textbf{Signal received at receiver.} When the signal arrives at the receiver, a series of operations are performed  to recover the data being transimitted. The incoming signal is downsampled and preamplified by the amplifier first. Then it passes through a matched filter with the same impulse responses at Eq.\eqref{pulse-shape} to reduce the noise. Next, in order to mitigate the distortion caused by transimission, an efficient recovery algorithm is needed. Our DOSnet is proposed at this end to compensates the linear and nonlinear distortion. 

Finally, the signal is downsampled again such that each symbol contains only one sample, and a classification/decision is carried out for each sample according to its distance to the points in the standard constellation $\cC$. Then the signal is classified into 16 grid points on the constellation. Furthermore, those grid points on the constellation are converted to a binary sequence by Gray code (See Fig.\ref{fig:graycode}). The original signal is recovered. \\
\begin{figure}[H]
    \centering
    \includegraphics[width = 0.65\linewidth]{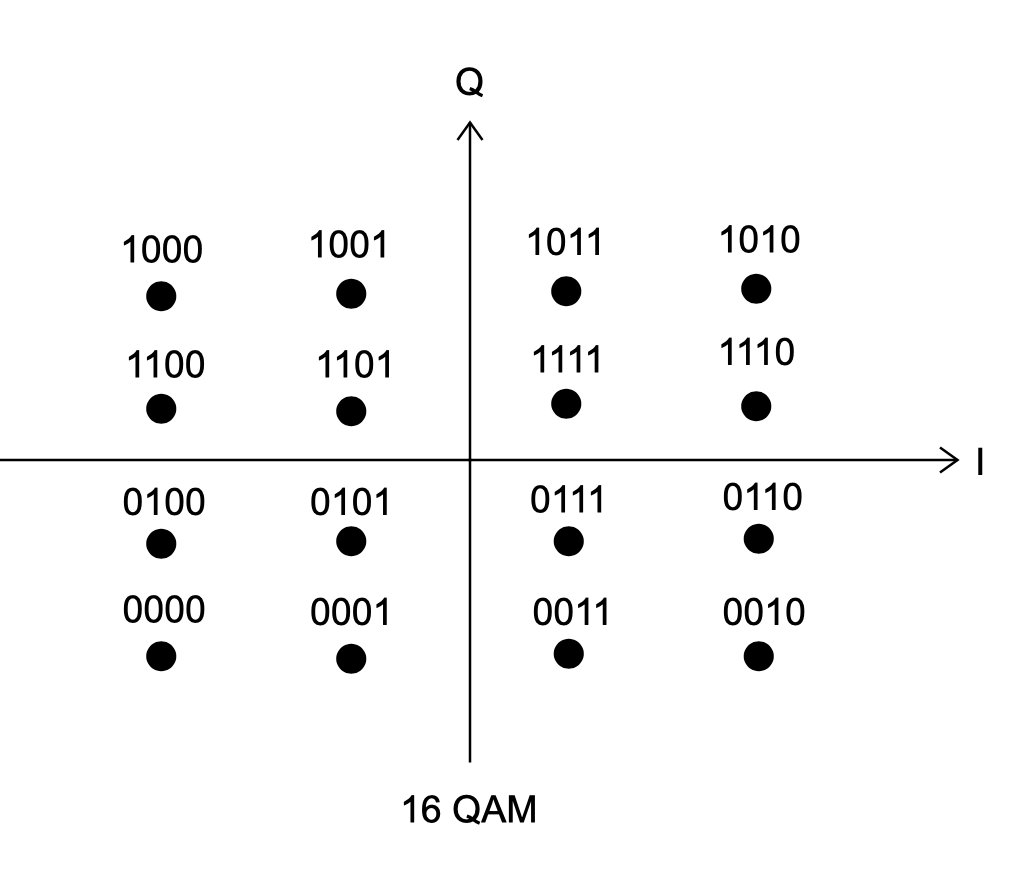}
    \caption{16 grids points on 16 QAM constellation and their corresponding gray codes.}
    \label{fig:graycode}
\end{figure}

%% file: main.bbl
\begin{thebibliography}{10}

\bibitem{Li2008ElectronicPO}
Xiaoxu Li, Xin Chen, Gilad Goldfarb, Eduardo~F. Mateo, Inwoong Kim, Fatih
  Yaman, and Guifang Li.
\newblock Electronic post-compensation of wdm transmission impairments using
  coherent detection and digital signal processing.
\newblock {\em Optics express}, 16 2:880--8, 2008.

\bibitem{Ip2008CompensationOD}
Ezra Ip and Joseph~M. Kahn.
\newblock Compensation of dispersion and nonlinear impairments using digital
  backpropagation.
\newblock {\em Journal of Lightwave Technology}, 26:3416--3425, 2008.

\bibitem{Asif2011DigitalBP}
Rameez Asif, Chien-Yu Lin, and Bernhard Schmauss.
\newblock Digital backward propagation: A technique to compensate fiber
  dispersion and non-linear impairments.
\newblock 2011.

\bibitem{Weideman1986SplitstepMF}
J.~Andre Weideman and Ben~M. Herbst.
\newblock Split-step methods for the solution of the nonlinear schro¨dinger
  equation.
\newblock {\em SIAM Journal on Numerical Analysis}, 23:485--507, 1986.

\bibitem{Hardin1973ApplicationOT}
Ronald~H. Hardin.
\newblock Application of the split-step fourier method to the numerical
  solution of nonlinear and variable coefficient wave equations.
\newblock {\em Siam Review}, 15:423, 1973.

\bibitem{strang1968construction}
Gilbert Strang.
\newblock On the construction and comparison of difference schemes.
\newblock {\em SIAM journal on numerical analysis}, 5(3):506--517, 1968.

\bibitem{Du2010ImprovedSC}
Liang~B. Du and Arthur~James Lowery.
\newblock Improved single channel backpropagation for intra-channel fiber
  nonlinearity compensation in long-haul optical communication systems.
\newblock {\em Optics express}, 18 16:17075--88, 2010.

\bibitem{Lin2010CompensationOT}
Chien-Yu Lin, Michael Holtmannspoetter, M.~Rameez Asif, and Bernhard Schmauss.
\newblock Compensation of transmission impairments by digital backward
  propagation for different link designs.
\newblock {\em 36th European Conference and Exhibition on Optical
  Communication}, pages 1--3, 2010.

\bibitem{Asif2011LogarithmicSB}
Rameez Asif, Chien-Yu Lin, Michael Holtmannspoetter, and Bernhard Schmauss.
\newblock Logarithmic step-size based digital backward propagation in n-channel
  112gbit/s/ch dp-qpsk transmission.
\newblock {\em 2011 13th International Conference on Transparent Optical
  Networks}, pages 1--4, 2011.

\bibitem{Fougstedt2017TimedomainDB}
Christoffer Fougstedt, Mikael Mazur, Lars~J. Svensson, Henrik Eliasson, Magnus
  Karlsson, and Per Larsson-Edefors.
\newblock Time-domain digital back propagation: Algorithm and finite-precision
  implementation aspects.
\newblock {\em 2017 Optical Fiber Communications Conference and Exhibition
  (OFC)}, pages 1--3, 2017.

\bibitem{Li2011ImplementationEN}
Lei Li, Zhenning Tao, Liang Dou, Weizhen Yan, Sho ichiro Oda, Takahito
  Tanimura, Takeshi Hoshida, and Jens~C. Rasmussen.
\newblock Implementation efficient nonlinear equalizer based on correlated
  digital backpropagation.
\newblock {\em 2011 Optical Fiber Communication Conference and Exposition and
  the National Fiber Optic Engineers Conference}, pages 1--3, 2011.

\bibitem{Yan2011LowCD}
Weizhen Yan, Zhenning Tao, Liang Dou, Lei Li, Sho ichiro Oda, Takahito
  Tanimura, Takeshi Hoshida, and Jens~C. Rasmussen.
\newblock Low complexity digital perturbation back-propagation.
\newblock {\em 2011 37th European Conference and Exhibition on Optical
  Communication}, pages 1--3, 2011.

\bibitem{Liang2015CorrelatedDB}
Xiaojun Liang and Shiva Kumar.
\newblock Correlated digital back propagation based on perturbation theory.
\newblock {\em Optics express}, 23 11:14655--65, 2015.

\bibitem{Rafique2011ImpactOS}
Danish Rafique and Andrew~D. Ellis.
\newblock Impact of signal-ase four-wave mixing on the effectiveness of digital
  back-propagation in 112 gb/s pm-qpsk systems.
\newblock {\em Optics express}, 19 4:3449--54, 2011.

\bibitem{Beck2020AnOO}
Christian Beck, Martin Hutzenthaler, Arnulf Jentzen, and Benno Kuckuck.
\newblock An overview on deep learning-based approximation methods for partial
  differential equations.
\newblock {\em ArXiv}, abs/2012.12348, 2020.

\bibitem{Raissi2019PhysicsinformedNN}
Maziar Raissi, Paris Perdikaris, and George~Em Karniadakis.
\newblock Physics-informed neural networks: A deep learning framework for
  solving forward and inverse problems involving nonlinear partial differential
  equations.
\newblock {\em J. Comput. Phys.}, 378:686--707, 2019.

\bibitem{Lu2019DeepONetLN}
Lu~Lu, Pengzhan Jin, and George~Em Karniadakis.
\newblock Deeponet: Learning nonlinear operators for identifying differential
  equations based on the universal approximation theorem of operators.
\newblock {\em ArXiv}, abs/1910.03193, 2019.

\bibitem{Han2018SolvingHP}
Jiequn Han, Arnulf Jentzen, and Weinan E.
\newblock Solving high-dimensional partial differential equations using deep
  learning.
\newblock {\em Proceedings of the National Academy of Sciences}, 115:8505 --
  8510, 2018.

\bibitem{Li2021FourierNO}
Zong-Yi Li, Nikola~B. Kovachki, Kamyar Azizzadenesheli, Burigede Liu, Kaushik
  Bhattacharya, Andrew Stuart, and Anima Anandkumar.
\newblock Fourier neural operator for parametric partial differential
  equations.
\newblock {\em ArXiv}, abs/2010.08895, 2021.

\bibitem{hastie2009elements}
Trevor Hastie, Robert Tibshirani, Jerome~H Friedman, and Jerome~H Friedman.
\newblock {\em The elements of statistical learning: data mining, inference,
  and prediction}, volume~2.
\newblock Springer, 2009.

\bibitem{fujisawa2021nonlinear}
Shinsuke Fujisawa, Fatih Yaman, Hussam~G Batshon, Massaki Tanio, Naoto Ishii,
  Chaoran Huang, Thomas~Ferreira De~Lima, Yoshihisa Inada, Paul~R Prucnal,
  Norifumi Kamiya, et~al.
\newblock Nonlinear impairment compensation using neural networks.
\newblock In {\em Optical Fiber Communication Conference}, pages M5F--1.
  Optical Society of America, 2021.

\bibitem{zhao2020low}
Yan Zhao, Xue Chen, Tao Yang, Liqian Wang, Danshi Wang, Zhiguo Zhang, and
  Sheping Shi.
\newblock Low-complexity fiber nonlinearity impairments compensation enabled by
  simple recurrent neural network with time memory.
\newblock {\em IEEE Access}, 8:160995--161004, 2020.

\bibitem{kamalov2018evolution}
Valey Kamalov, Ljupcho Jovanovski, Vijay Vusirikala, Shaoliang Zhang, Fatih
  Yaman, Kohei Nakamura, Takanori Inoue, Eduardo Mateo, and Yoshihisa Inada.
\newblock Evolution from 8qam live traffic to ps 64-qam with neural-network
  based nonlinearity compensation on 11000 km open subsea cable.
\newblock In {\em Optical Fiber Communication Conference}, pages Th4D--5.
  Optica Publishing Group, 2018.

\bibitem{hager_deep_2018}
Christian Häger and Henry~D. Pfister.
\newblock Deep {Learning} of the {Nonlinear} {Schrödinger} {Equation} in
  {Fiber}-{Optic} {Communications}.
\newblock In {\em 2018 {IEEE} {International} {Symposium} on {Information}
  {Theory} ({ISIT})}, pages 1590--1594, June 2018.
\newblock ISSN: 2157-8117.

\bibitem{Hger2018NonlinearIM}
Christian H{\"a}ger and Henry~D. Pfister.
\newblock Nonlinear interference mitigation via deep neural networks.
\newblock {\em 2018 Optical Fiber Communications Conference and Exposition
  (OFC)}, pages 1--3, 2018.

\bibitem{fan_advancing_2020}
Qirui Fan, Gai Zhou, Tao Gui, Chao Lu, and Alan Pak~Tao Lau.
\newblock Advancing theoretical understanding and practical performance of
  signal processing for nonlinear optical communications through machine
  learning.
\newblock {\em Nature Communications}, 11(1):1--11, July 2020.
\newblock Number: 1 Publisher: Nature Publishing Group.

\bibitem{quispel2002splitting}
G~Reinout~W Quispel.
\newblock Splitting methods.
\newblock {\em Acta Numerica 2002: Volume 11}, (11):341--434, 2002.

\bibitem{simonyan2014very}
Karen Simonyan and Andrew Zisserman.
\newblock Very deep convolutional networks for large-scale image recognition.
\newblock {\em arXiv preprint arXiv:1409.1556}, 2014.

\bibitem{lang2012differential}
Serge Lang.
\newblock {\em Differential and Riemannian manifolds}, volume 160.
\newblock Springer Science \& Business Media, 2012.

\bibitem{glorot2011deep}
Xavier Glorot, Antoine Bordes, and Yoshua Bengio.
\newblock Deep sparse rectifier neural networks.
\newblock In {\em Proceedings of the fourteenth international conference on
  artificial intelligence and statistics}, pages 315--323. JMLR Workshop and
  Conference Proceedings, 2011.

\bibitem{lecun2012efficient}
Yann~A LeCun, L{\'e}on Bottou, Genevieve~B Orr, and Klaus-Robert M{\"u}ller.
\newblock Efficient backprop.
\newblock In {\em Neural networks: Tricks of the trade}, pages 9--48. Springer,
  2012.

\bibitem{bonfiglioli2011topics}
Andrea Bonfiglioli and Roberta Fulci.
\newblock {\em Topics in noncommutative algebra: the theorem of Campbell,
  Baker, Hausdorff and Dynkin}, volume 2034.
\newblock Springer Science \& Business Media, 2011.

\bibitem{jing2017tunable}
Li~Jing, Yichen Shen, Tena Dubcek, John Peurifoy, Scott Skirlo, Yann LeCun, Max
  Tegmark, and Marin Solja{\v{c}}i{\'c}.
\newblock Tunable efficient unitary neural networks (eunn) and their
  application to rnns.
\newblock In {\em International Conference on Machine Learning}, pages
  1733--1741. PMLR, 2017.

\bibitem{kiani2022projunn}
Bobak Kiani, Randall Balestriero, Yann Lecun, and Seth Lloyd.
\newblock projunn: efficient method for training deep networks with unitary
  matrices.
\newblock {\em arXiv preprint arXiv:2203.05483}, 2022.

\bibitem{du2020phase}
Qiang Du and Xiaobing Feng.
\newblock The phase field method for geometric moving interfaces and their
  numerical approximations.
\newblock {\em Handbook of Numerical Analysis}, 21:425--508, 2020.

\bibitem{kao_dielectric-fibre_1966}
K.C. Kao and G.A. Hockham.
\newblock Dielectric-fibre surface waveguides for optical frequencies.
\newblock {\em Proceedings of the Institution of Electrical Engineers},
  113(7):1151--1158, July 1966.

\bibitem{agrawal_fiber-optic_2010}
Govind~P. Agrawal.
\newblock {\em Fiber-{Optic} {Communication} {Systems}}.
\newblock Wiley, New York, 4 edition edition, October 2010.

\bibitem{he_deep_2016}
K.~He, X.~Zhang, S.~Ren, and J.~Sun.
\newblock Deep {Residual} {Learning} for {Image} {Recognition}.
\newblock In {\em 2016 {IEEE} {Conference} on {Computer} {Vision} and {Pattern}
  {Recognition} ({CVPR})}, pages 770--778, June 2016.
\newblock ISSN: 1063-6919.

\bibitem{graves2005bidirectional}
Alex Graves, Santiago Fern{\'a}ndez, and J{\"u}rgen Schmidhuber.
\newblock Bidirectional lstm networks for improved phoneme classification and
  recognition.
\newblock In {\em International conference on artificial neural networks},
  pages 799--804. Springer, 2005.

\bibitem{paszke2019pytorch}
Adam Paszke, Sam Gross, Francisco Massa, Adam Lerer, James Bradbury, Gregory
  Chanan, Trevor Killeen, Zeming Lin, Natalia Gimelshein, Luca Antiga, et~al.
\newblock Pytorch: An imperative style, high-performance deep learning library.
\newblock {\em Advances in neural information processing systems}, 32, 2019.

\bibitem{saxe2013exact}
Andrew~M Saxe, James~L McClelland, and Surya Ganguli.
\newblock Exact solutions to the nonlinear dynamics of learning in deep linear
  neural networks.
\newblock {\em arXiv preprint arXiv:1312.6120}, 2013.

\end{thebibliography}
